\pdfoutput=1
\documentclass[11pt]{article}
\usepackage[letterpaper,margin=1in]{geometry}
\usepackage[T1]{fontenc}
\usepackage{amsmath}
\usepackage{amssymb}
\usepackage{amsthm}
\usepackage{mathtools}
\usepackage{newtxtext}
\usepackage{newtxmath}
\usepackage{bm}
\usepackage{booktabs}
\usepackage{threeparttable}
\usepackage{array}
\usepackage{graphicx}
\usepackage{float}
\usepackage{placeins}
\usepackage{algorithm}
\usepackage[noend]{algpseudocode}
\usepackage{tikz}
\usetikzlibrary{arrows.meta,calc,positioning}
\usepackage{pgfplots}
\usepackage{xcolor}
\usepackage{caption}
\usepackage{natbib}
\usepackage{hyperref}
\usepackage[nameinlink,capitalize]{cleveref}
\crefname{equation}{Equation}{Equations}
\Crefname{equation}{Equation}{Equations}

\pgfplotsset{compat=1.16}
\hypersetup{
  colorlinks=true,
  linkcolor=blue!45!black,
  citecolor=blue!45!black,
  urlcolor=blue!45!black
}
\newtheorem{theorem}{Theorem}
\newtheorem{lemma}{Lemma}
\newtheorem{proposition}{Proposition}
\newtheorem{corollary}{Corollary}

\makeatletter
\providecommand*{\theHALG@line}{\thealgorithm.\arabic{ALG@line}}
\makeatother

\providecommand{\R}{\mathbb{R}}
\providecommand{\conv}{\operatorname{conv}}
\providecommand{\norm}[1]{\lVert #1\rVert}
\providecommand{\Q}{\mathcal{Q}}

\providecommand{\vct}[1]{\bm{#1}}
\DeclareMathOperator*{\argmin}{arg\,min}

\newcommand{\RealEndpointPrimaryRuns}{570}

\newcommand{\JuliaRealCRSpeed}{1.38}

\newcommand{\MatlabRealCRSpeed}{1.15}

\newcommand{\ROPrimaryObjectiveSolves}{11,520}
\newcommand{\ROSimplexObjectiveSolves}{1,440}

\newcommand{\ROJuliaCRSpeedupMany}{1.85}
\newcommand{\ROJuliaCRSpeedupManyLower}{1.74}
\newcommand{\ROJuliaCRSpeedupManyUpper}{1.98}

\newcommand{\ROMatlabCRSpeedupMany}{1.48}
\newcommand{\ROMatlabCRSpeedupManyLower}{1.39}
\newcommand{\ROMatlabCRSpeedupManyUpper}{1.57}

 \newcommand{\ROPortfolioObjectiveSolves}{480}

\newcommand{\ROPortfolioJuliaCRSpeedup}{0.90}
\newcommand{\ROPortfolioMatlabCRSpeedup}{0.79}

\newcommand{\ROPortfolioHighsSpeedup}{9.32}

 \newcommand{\JuliaAffineAssociation}{0.83}
\newcommand{\JuliaAffineAssociationLower}{0.76}
\newcommand{\JuliaAffineAssociationUpper}{0.88}
\newcommand{\JuliaScanAssociation}{0.81}
\newcommand{\JuliaScanAssociationLower}{0.73}
\newcommand{\JuliaScanAssociationUpper}{0.86}
\newcommand{\MatlabAffineAssociation}{0.74}
\newcommand{\MatlabAffineAssociationLower}{0.64}
\newcommand{\MatlabAffineAssociationUpper}{0.81}
\newcommand{\MatlabScanAssociation}{0.84}
\newcommand{\MatlabScanAssociationLower}{0.75}
\newcommand{\MatlabScanAssociationUpper}{0.90}
 \newcommand{\EToEJuliaPCRGM}{1.43}
\newcommand{\EToEMatlabPCRGM}{1.54}
\newcommand{\EToEJuliaHighsCRGM}{3.48}
\newcommand{\EToEMatlabHighsCRGM}{2.36}
\newcommand{\EToETimedRows}{720}
\newcommand{\EToERowsPerSolver}{240}
\newcommand{\EToERegressionRows}{24}
\newcommand{\EToEBoundaryCases}{3}
\newcommand{\EToEMaxObjectiveDisagreement}{1.483\times10^{-13}}

\hypersetup{
  pdftitle={Row-Polar LP-Newton for Linear Programming with Corral Repair},
  pdfsubject={Optimization Online research report},
  pdfkeywords={linear programming; row-polar formulation; LP reoptimization; changing objectives; LP-Newton method; corral repair; minimum-norm point},
  pdfauthor={Yanfei Li, Yuki Matsuno, and Jianming Shi}
}

\title{Row-Polar LP-Newton for Linear Programming with Corral Repair}
\author{Yanfei Li\thanks{Shandong University, China.}
\and Yuki Matsuno\thanks{Seikei University, Japan.}
\and Jianming Shi\thanks{Shandong University, China; Seikei University,
Japan; and Tokyo University of Science, Japan. Corresponding author:
\texttt{shi@rs.tus.ac.jp}.}}
\date{September 3, 2026 (Revised September 8, 2026)}

\begin{document}
\maketitle

\begin{abstract}
LP-Newton solves a linear program through a sequence of nearest-point problems.
Given an interior feasible point for an inequality-form LP, we construct the
compact hull formed by the origin and its normalized constraint rows.  Polar LP-Newton (P-LPN)
follows the objective ray to the boundary of this row-polar hull, where it
recovers a primal--dual optimum or a recession direction proving unboundedness.
For rational row-polar data, we prove an outer-iteration bound quadratic in
dimension and linear in binary input length, excluding inner Wolfe work.

Restarting Wolfe at every outer iteration discards its terminal corral even
though the hull is unchanged and the next target lies on the same ray.
Corral-repair P-LPN (CR-P-LPN) instead repairs that corral and uses it to start
the next projection.  Verification over the full hull preserves P-LPN's
projections, outer targets, and LP conclusion in exact arithmetic.

Experiments in Julia and MATLAB show where repair helps.  After common
initialization, median P-LPN times across 180 single-LP tests are 1.62 and 1.55
times the corresponding CR-P-LPN times.  On three problems with
5,000--100,000 rows, the geometric-mean factors are 2.03 and 1.46.  A separate
end-to-end study includes initialization on 19 application-derived LPs.
CR-P-LPN has lower median total time than the faster cold HiGHS mode---dual
simplex or the interior-point method with crossover---on 18 workloads in each
language; the geometric-mean HiGHS/CR-P-LPN ratios are 4.93 and 3.06.
Component comparisons are consistent with terminal-corral reuse as the useful
component, whereas candidate-first ordering alone shows no consistent gain.
Repair is not always beneficial: it is slower on Klee--Minty tests,
and an author-implemented simplex is faster on the retained Netlib models.
\end{abstract}

\medskip

\section{Introduction}\label{sec:introduction}

LP-Newton approaches linear programming through repeated nearest-point
projections.  We develop a version for inequality-form LPs with many
constraints.  A known interior feasible point provides positive slacks; after
translation and row normalization, the constraints become
\(A\vct{x}\leq\mathbf 1\).  The origin together with the rows of \(A\) forms a
compact row-polar hull \(\Q\), whether or not the original feasible region is
bounded.  Polar LP-Newton (P-LPN) searches for the endpoint of the objective ray
in \(\Q\).  A positive endpoint recovers a primal--dual optimum, while a zero
endpoint supplies a recession direction and proves unboundedness.

Each P-LPN iteration projects a point on the objective ray onto \(\Q\).  Wolfe's
algorithm ends this computation at a \emph{corral}: an affinely independent
active set whose affine minimizer is a strictly positive convex combination of
its generators.  Unless P-LPN has already stopped, the projection residual
defines a supporting hyperplane; where that hyperplane meets the ray is the next
target.  The target moves, but the hull does not.  This simple observation is
the reason to retain Wolfe's terminal state.

Our corral-repair method, CR-P-LPN, adjusts the preceding corral to the new
target and restarts Wolfe from that set.  Before accepting the computed point,
the method either establishes that the target belongs to \(\Q\) or checks
Wolfe's inequality for every generator.  If the warm computation stops before
completing this verification, unrestricted Wolfe recomputes the projection
from a single generator.  Corral repair therefore changes only the
initialization of each projection.  In exact
arithmetic, P-LPN and CR-P-LPN generate the same outer targets and return the
same optimum or unboundedness conclusion.

The experiments separate projection performance from initialization cost.
Phase-II tests begin with the same normalized problem and compare P-LPN with
CR-P-LPN; component variants isolate corral repair.  End-to-end tests add
artificial Phase I, normalization, inverse transformation, and verification,
and compare both row-polar methods with cold HiGHS on the same original data.
Additional tests identify cases in which repair costs more than the Wolfe work
it saves.

The main contributions are:

\begin{enumerate}
  \item \emph{Row-polar formulation and outer-iteration theory.}  We express an
        inequality-form LP with a known interior point as a boundary search on a
        compact hull of normalized constraint rows.  The endpoint recovers an
        optimum or proves unboundedness.  For rational row-polar data, P-LPN
        needs \(O(n^2\mathcal L)\) outer iterations, excluding inner Wolfe work;
        \(\mathcal L\) is the binary length of the row-polar input.
  \item \emph{Verified terminal-corral repair.}  CR-P-LPN carries a terminal
        Wolfe corral from one projection to the next, repairs it, and verifies
        the result over the full hull.  It therefore computes the same outer
        targets and LP outcome as P-LPN in exact arithmetic and inherits its
        outer-iteration bound.  We also bound the arithmetic work and storage
        of an implementation that retrieves inactive generators from \(A\) and
        stores only active ones.
  \item \emph{Computational evidence.}  Paired Julia and MATLAB experiments
        compare P-LPN and CR-P-LPN after common initialization and in
        end-to-end solves; component variants examine repair separately.
        On 18 of 19 application-derived workloads in each language, CR-P-LPN
        has lower median total time than the faster cold HiGHS mode (dual
        simplex or the interior-point method with crossover).  Other tests
        reveal slowdowns as well as gains.  A secondary experiment studies
        corral transfer across a finite sequence of objectives.
\end{enumerate}

Section~2 reviews related work; Sections~3--6 develop and analyze P-LPN and
CR-P-LPN; Section~7 reports the experiments; and Sections~8--9 discuss the
results and conclude.  Appendices~\ref{app:outer-details}--\ref{app:numerical-details}
make the report self-contained by giving the counting details, optional
projection analyses, two-stage initialization, verification rules, comparator
settings, and complete supporting tables.

\section{Related Work}\label{sec:related}

The closest predecessors are LP-Newton methods.  \citet{FujishigeEtAl2009}
map box-constrained equality-form LPs to a zonotope and prove finite
termination.  Their analysis follows terminal corrals through successive
projections \citep[Theorem~3.11]{FujishigeEtAl2009}, although the algorithm
restarts each projection.  Later variants treat standard-form LPs through a
related cone \citep{KitaharaEtAl2013}, add binary search to zonotope projections
\citep{KitaharaSukegawa2019}, and extend LP-Newton to conic programs through a
semi-infinite representation \citep{TanakaOkuno2021}.  P-LPN keeps the repeated
projection and radial search, but changes the projection domain to the compact
hull of normalized inequality rows.  Its \(O(n^2\mathcal L)\) outer-iteration
bound builds on the scalar distance interpretation in
\citet[Remark~3.5]{FujishigeEtAl2009}.

The earlier random experiments found zonotope LP-Newton competitive with MATLAB
\texttt{linprog} only at the smallest tested size
\citep[Table~1]{FujishigeEtAl2009}.  Those timings come from a different model,
implementation, and machine.  Our current general-purpose reference is HiGHS,
which is also the default dual-simplex backend of MATLAB R2025b
\texttt{linprog}; Section~7 times cold HiGHS and the two row-polar methods on
common data and hardware.

LP sensitivity and reoptimization motivate retaining solver state when the
feasible region stays fixed.  When only the objective changes, simplex can
start from the preceding basis \citep{Dantzig1963}.  In the secondary extension
studied here, RO-CR-P-LPN offers the preceding terminal corral to the first
projection for the next objective.  In both methods, the retained state is
only a starting point; optimality must be established again.

Within each projection, P-LPN uses Wolfe's minimum-norm-point algorithm, which
returns both the projection and its terminal corral \citep{Wolfe1976}.
Update-and-stabilize modifies an active set while solving one such problem
\citep{FujishigeEtAl2025}; CR-P-LPN carries a terminal corral to the next
projection problem.  Section~5 briefly examines two suggestions from
\citet{Fujishige2019}: weighting the projection metric and testing selected
generators first.  The final method retains Euclidean projection and uses the
second suggestion only to choose the order of verification after repair.
Appendix~\ref{app:projection-details} contrasts this construction with dual minimum-norm,
nearest-pair, and affine-restriction methods
\citep{FujishigeZhan1990,FujishigeZhan1992,FujishigeEtAl1994}.

Endpoint recovery rests on classical polarity and LP duality
\citep{Schrijver1986}.  A nearby geometric idea is the projective cutting-plane
method, which reaches the boundary of an inequality-described polytope by a
directional projection from an interior point \citep{Porumbel2022}.  P-LPN
instead reaches the boundary of a generator-described row-polar hull through a
sequence of Euclidean nearest-point projections.

\section{Row-Polar Formulation and LP Recovery}\label{sec:formulation}

\smallskip
\noindent\textbf{Problem and normalized formulation.}
Consider the original LP
\[
 (\mathrm P^0)\qquad
 \max\{(\vct{c}^{\,0})^T\vct{x}^{\,0}:
 A^0\vct{x}^{\,0}\leq\vct{b}^{\,0}\},
\]
where \(A^0\in\R^{m\times n}\), \(\vct{b}^{\,0}\in\R^m\), and
\(\vct{c}^{\,0}\in\R^n\setminus\{\vct{0}\}\).  Throughout,
\(\norm{\cdot}\) denotes the Euclidean norm.  Suppose that a strictly feasible
point \(\bar{\vct{x}}^{\,0}\) is given.  Define
\[
 \vct{s}^{\,0}:=\vct{b}^{\,0}-A^0\bar{\vct{x}}^{\,0}>\vct{0},
 \qquad S:=\operatorname{Diag}(\vct{s}^{\,0}),
\]
and set
\[
 \vct{x}:=\vct{x}^{\,0}-\bar{\vct{x}}^{\,0},\qquad
 A:=S^{-1}A^0,\qquad
 \vct{c}:=\frac{\vct{c}^{\,0}}{\norm{\vct{c}^{\,0}}}.
\]
Mathematically, the original LP \((\mathrm P^0)\) is equivalent to
\begin{equation}\label{eq:primal}
 (\mathrm P)\qquad
 \max\{\vct{c}^T\vct{x}:A\vct{x}\leq\mathbf 1\},
 \qquad \norm{\vct{c}}=1.
\end{equation}
An optimum \(\vct{x}^*\) of \((\mathrm P)\) gives the original solution
\(\vct{x}^{\,0*}:=\bar{\vct{x}}^{\,0}+\vct{x}^*\).  We may therefore analyze
\((\mathrm P)\), whose origin is strictly feasible.  Appendices~\ref{app:outer-details} and~\ref{app:verification} give the remaining inverse transformations and
original-scale verification rules.  When no strictly feasible point is known,
the artificial Phase-I procedure in Appendix~\ref{app:initialization} either finds
one, reduces the problem to its affine hull, or detects infeasibility.

Let \(\vct{a}_i^T\) denote row \(i\) of the normalized matrix \(A\), and define
\begin{equation}\label{eq:polarvertices}
  \Q:=\conv\{\vct{0},\vct{a}_1,\ldots,\vct{a}_m\}.
\end{equation}
We call these \(m+1\) points the generator list.  Define
\begin{equation}\label{eq:rho}
  \rho^*:=\max\{\rho\geq0:\rho\vct{c}\in\Q\}.
\end{equation}
Compactness of \(\Q\) and the inclusion \(\vct{0}\in\Q\) make the maximum
well defined.

\begin{figure}[H]
\centering
\resizebox{0.98\textwidth}{!}{\vbox{\vskip3pt\hbox{%
\begin{tikzpicture}[>=stealth]
\tikzset{stage/.style={draw=black!55,fill=black!3,align=center,
  minimum height=1.55cm,text width=2.85cm,font=\scriptsize,inner sep=4pt}}
\node[stage] (input) at (-0.55,0)
  {\textbf{Normalized LP}\\
   \(A\vct{x}\leq\mathbf 1\), \(\norm{\vct{c}}=1\)};
\node[stage] (hull) at (3.60,0)
  {\textbf{Row-polar hull}\\
   \(\Q=\conv\{\vct{0},\vct{a}_1,\ldots,\vct{a}_m\}\)};
\node[stage] (project) at (7.75,0)
  {\textbf{P-LPN projections}\\project objective-ray targets onto \(\Q\);
   CR-P-LPN repairs the preceding corral};
\node[stage] (recover) at (11.90,0)
  {\textbf{Radial endpoint}\\\(\rho^*>0\): recover
   primal--dual optima\\\(\rho^*=0\): recover a recession direction};
\draw[->,thick,black!65] (input) -- node[above,font=\tiny]{form} (hull);
\draw[->,thick,black!65] (hull) -- node[above,font=\tiny]{project} (project);
\draw[->,thick,black!65] (project) -- node[above,font=\tiny]{recover} (recover);
\end{tikzpicture}}}}
\caption{Row-polar LP-Newton workflow from the normalized problem.  During one LP solve,
the hull and objective ray remain fixed;
only the target radius changes between projections.
\label{fig:row-polar-workflow}}
\begin{minipage}{0.98\textwidth}
\scriptsize The endpoint formulas are given in
\cref{thm:polar-equivalence}; corral repair is developed in
\cref{sec:corralreuse}.
\end{minipage}
\end{figure}
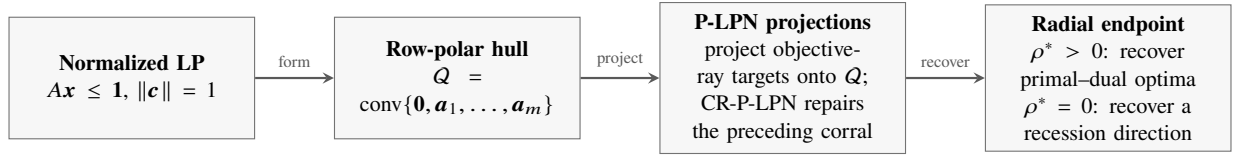
\FloatBarrier

For a set \(\mathcal S\subseteq\R^n\), write
\[
 \mathcal S^\circ:=
 \{\vct{q}:\vct{q}^T\vct{w}\leq1\;\forall\vct{w}\in\mathcal S\},
 \qquad
 \{\vct{x}:A\vct{x}\leq\mathbf 1\}^\circ=\Q.
\]
The second relation is the standard polar formula for an intersection of
halfspaces; in particular, the origin in \cref{eq:polarvertices} is essential.

\smallskip
\noindent\textbf{Radial endpoint and LP recovery.}

Polarity and a supporting normal recover an optimal solution directly.  To verify optimality through the standard primal--dual conditions, consider the
normalized dual
\[
 \min\{\mathbf 1^T\vct{u}:A^T\vct{u}=\vct{c},\ \vct{u}\geq\vct{0}\}.
\]
P-LPN never optimizes in the \(m\)-dimensional dual space.  It instead finds the
largest \(\rho\) for which \(\rho\vct{c}\in\Q\subset\R^n\); at a positive
endpoint, Wolfe's row weights rescale to a dual optimum.

\begin{theorem}[Radial outcome and solution recovery for \((\mathrm P)\)]\label{thm:polar-equivalence}
Under the normalization above, exactly one of the following holds.
\begin{enumerate}
  \item If \(\rho^*>0\), then \((\mathrm P)\) has the finite optimal value
  \begin{equation}\label{eq:value}
    \max\{\vct{c}^T\vct{x}:A\vct{x}\leq\mathbf 1\}
    =\frac{1}{\rho^*}.
  \end{equation}
  Define \(\vct{q}^*:=\rho^*\vct{c}\).  If
  \(\vct{d}^T\vct{q}\leq\vct{d}^T\vct{q}^*\) supports \(\Q\) at
  \(\vct{q}^*\), with
  \(\vct{d}^T\vct{q}^*>0\), then
  \begin{equation}\label{eq:primalrecover}
    \vct{x}^*:=\frac{\vct{d}}{\vct{d}^T\vct{q}^*}
  \end{equation}
  is optimal for \((\mathrm P)\).  If
  \(\vct{q}^*=\sum_{i=1}^m\lambda_i\vct{a}_i\), with
  \(\lambda_i\geq0\) and \(\sum_i\lambda_i=1\), define
  \begin{equation}\label{eq:dualrecover}
    u_i^*:=\frac{\lambda_i}{\rho^*}
    \quad(i=1,\ldots,m).
  \end{equation}
  Then, for any such \(\vct{d}\) and weights \(\lambda_i\),
  \((\vct{x}^*,\vct{u}^*)\) is a primal--dual optimum for \((\mathrm P)\).
  \item If \(\rho^*=0\), then \((\mathrm P)\) is unbounded above.
  Moreover, a vector \(\vct{d}\) satisfying
  \(\vct{d}^T\vct{q}\leq0\) for every \(\vct{q}\in\Q\) and
  \(\vct{d}^T\vct{c}>0\) gives the recession direction
  \(\vct{r}:=\vct{d}/(\vct{c}^T\vct{d})\), for which
  \(A\vct{r}\leq\vct{0}\) and \(\vct{c}^T\vct{r}=1\).
\end{enumerate}
\end{theorem}

\begin{proof}
Suppose \(\rho^*>0\), and choose a supporting normal \(\vct{d}\) with
\(\vct{d}^T\vct{q}^*>0\) at the nonzero endpoint.  For every feasible
\(\vct{x}\), polarity gives
\((\rho^*\vct{c})^T\vct{x}\leq1\), hence
\(\vct{c}^T\vct{x}\leq1/\rho^*\).  Dividing the supporting inequality by
\(\vct{d}^T\vct{q}^*\) shows that
  \(\vct{d}/(\vct{d}^T\vct{q}^*)\) satisfies every constraint.  Because
  \(\vct{q}^*=\rho^*\vct{c}\), this point attains the bound, proving
  \cref{eq:value,eq:primalrecover}.

The origin cannot have positive weight in a convex representation of the
positive endpoint; otherwise rescaling the remaining weights would give a
farther point of \(\Q\) on the same ray.  Hence the stated row-weight
representation exists.  The normalized row representation and
\cref{eq:dualrecover} give
\[
  \vct{u}^*\geq\vct{0},\qquad
  A^T\vct{u}^*
  =\frac{1}{\rho^*}\sum_i\lambda_i\vct{a}_i
  =\vct{c}.
\]
Moreover, \(\mathbf 1^T\vct{u}^*=1/\rho^*=\vct{c}^T\vct{x}^*\).
Thus primal feasibility, dual feasibility, and equality of the normalized
objectives certify optimality by weak duality.

If \(\rho^*=0\), strict separation of \(\vct{c}\) from the conic
hull of \(\Q\) gives a vector \(\vct{d}\) with the stated signs.
Since \(\vct{a}_i^T\vct{d}\leq0\), we have
\(A\vct{d}\leq\vct{0}\).  Dividing by the positive number
\(\vct{c}^T\vct{d}\) gives the stated \(\vct{r}\).  Because
\(A\vct{r}\leq\vct{0}\) and \(\vct{c}^T\vct{r}=1\), every
\(t\vct{r}\), \(t\geq0\), is feasible for \((\mathrm P)\), and its objective
 tends to infinity.\nobreak\hfill
\end{proof}

\smallskip
\noindent\textbf{A direct endpoint formulation.}
Equivalently, the radial endpoint satisfies
\begin{equation}\label{eq:direct-radial-qp}
 \rho^*=\max_{\rho,\vct{\lambda}}
 \left\{\rho:
 \sum_{i=1}^m\lambda_i\vct{a}_i=\rho\vct{c},\
 \sum_{i=1}^m\lambda_i\leq1,\
 \vct{\lambda}\geq\vct{0},\ \rho\geq0\right\}.
\end{equation}
The omitted weight \(1-\sum_i\lambda_i\) is assigned to the origin, so the
constraints in \cref{eq:direct-radial-qp} are precisely the condition
\(\rho\vct{c}\in\Q\).  Solving this LP can recover the endpoint and its row
weights, but not Wolfe's projection path or terminal corral.  That state is what
P-LPN obtains from \cref{eq:projection}; when the endpoint is zero, P-LPN also
uses the separating normal from \cref{thm:polar-equivalence}.

\smallskip
\noindent\textbf{Why the row-polar form is useful.}
The row-polar form serves three purposes.  Its radial endpoint distinguishes a
finite optimum from unboundedness; the same compact hull is used at every outer
iteration, which makes Wolfe state reusable; and products with \(A\) and
\(A^T\) provide row scores and generator combinations without a second matrix
of generator columns.  \Cref{sec:corralreuse} makes the last point explicit.

\section{P-LPN: Algorithm and Outer-Iteration Bounds}\label{sec:plpn}

P-LPN approaches the objective-ray endpoint from outside \(\Q\).  It projects a
ray target onto the hull and uses the resulting supporting hyperplane to choose
the next target.
Put \(\mathcal I:=\{0,\ldots,m\}\), \(\vct{v}_0:=\vct{0}\), and
\(\vct{v}_i:=\vct{a}_i\) for \(i\geq1\).  Choose
\(\rho_0>\max_{i\in\mathcal I}\vct{c}^T\vct{v}_i\), so
\(\rho_0\vct{c}\notin\Q\).  At outer iteration \(k\), let
\begin{equation}\label{eq:projection}
  \vct{y}_k:=\rho_k\vct{c},\qquad
  \vct{z}_k\in\argmin_{\vct{q}\in\Q}
  \tfrac12\norm{\vct{q}-\vct{y}_k}^2,
  \qquad \vct{d}_k:=\vct{y}_k-\vct{z}_k,
  \qquad h_k:=\vct{d}_k^T\vct{z}_k.
\end{equation}
Projection optimality gives
\(h_k=\max_{\vct{q}\in\Q}\vct{d}_k^T\vct{q}\), so
\(\vct{d}_k^T\vct{q}=h_k\) supports \(\Q\) at \(\vct{z}_k\).  Its
intersection with the objective ray is the next target:
\begin{equation}\label{eq:newtonupdate}
  \rho_{k+1}:=\frac{h_k}{\vct{d}_k^T\vct{c}},
  \qquad \vct{y}_{k+1}:=\rho_{k+1}\vct{c}.
\end{equation}

The call
\((\vct{z},C,\vct{\lambda}_C):=\mathrm{Wolfe}(\vct{y};\mathcal I)\)
returns the unique projection, its terminal corral, and positive convex
weights on that corral.  Thus the indexed generators are affinely independent,
\(\vct{z}=\sum_{i\in C}\lambda_i\vct{v}_i\), and
\(\sum_{i\in C}\lambda_i=1\).  Wolfe alternates generator insertions under its
linear-optimization rule with minor cycles that delete zero-weight generators.

Recovery is handled outside the displayed outer loop.  The routine
\(\mathrm{VerifyFinite}\) receives
\((\rho,\vct{d},h,C,\vct{\lambda}_C)\), extends the corral weights by zeros,
and applies the recovery formulas in
\cref{thm:polar-equivalence}.  It then tests primal feasibility, dual
feasibility, and equality of the normalized objectives, and returns
\(\mathrm{Optimal}\) or \(\mathrm{OptTestFailed}\) together with
\((\vct{x}^*,\vct{u}^*)\).
The companion routine \(\mathrm{VerifyRecession}(\vct{d})\) forms the recession direction,
tests the normalized recession conditions, and returns \(\mathrm{Unbounded}\)
or \(\mathrm{RecTestFailed}\).  Because
\(\vct{y}_0\notin\Q\), the \(\vct{z}_k=\vct{y}_k\) branch cannot occur at
\(k=0\), so no branch reads a preceding supporting hyperplane before one exists.
An outer iteration is complete once Wolfe has returned the current projection
and P-LPN has either stopped or chosen the next target.  Inner major and minor
steps, affine solves, and generator tests are not part of this count; Section~6
accounts for them separately.
\begin{algorithm}[!t]
\caption{Polar LP-Newton (P-LPN)}
\label{alg:plpn}
\footnotesize
\begin{algorithmic}[1]
\Require normalized problem \((\mathrm P)\), index set \(\mathcal I\), and
generator access through \(A\)
\State \(\rho_0\in
\bigl(\max_{i\in\mathcal I}\vct{c}^T\vct{v}_i,
+\infty\bigr)\), \(k\gets0\)
\Loop
  \State \(\vct{y}_k\gets\rho_k\vct{c}\),
  \((\vct{z}_k,C_k,\vct{\lambda}^{\,k}_{C_k})
  \gets\mathrm{Wolfe}(\vct{y}_k;\mathcal I)\)
  \If{\(\vct{z}_k=\vct{y}_k\)}
    \If{\(\rho_k=0\)}
      \State \Return \(\mathrm{VerifyRecession}(\vct{d}_{k-1})\)
    \EndIf
    \State \Return
    \(\mathrm{VerifyFinite}(\rho_k,\vct{d}_{k-1},h_{k-1},
    C_k,\vct{\lambda}^{\,k}_{C_k})\)
  \EndIf
  \State \(\vct{d}_k\gets\vct{y}_k-\vct{z}_k\), \quad
  \(h_k\gets\vct{d}_k^T\vct{z}_k
  =\max_{i\in\mathcal I}\vct{d}_k^T\vct{v}_i\)
  \If{\(h_k=0\)}
    \State \Return \(\mathrm{VerifyRecession}(\vct{d}_k)\)
  \EndIf
  \State \(\rho_{k+1}\gets
  h_k/(\vct{d}_k^T\vct{c})\), \quad \(k\gets k+1\)
\EndLoop
\end{algorithmic}
\end{algorithm}

\begin{lemma}[Valid outer update]\label{lem:update}
If \(\rho_k>\rho^*\), then
\(\vct{d}_k^T\vct{c}>0\) and
\begin{equation}\label{eq:bracket}
  \rho^*\leq\rho_{k+1}<\rho_k.
\end{equation}
\end{lemma}

\begin{proof}
Because \(\vct{q}^*=\rho^*\vct{c}\in\Q\), the projection
inequality gives
\(\rho^*\vct{d}_k^T\vct{c}\leq h_k\).
On the other hand,
\begin{equation*}
  \rho_k\vct{d}_k^T\vct{c}-h_k
  =\vct{d}_k^T(\vct{y}_k-\vct{z}_k)
  =\norm{\vct{y}_k-\vct{z}_k}^2>0.
\end{equation*}
Hence
\((\rho_k-\rho^*)\vct{d}_k^T\vct{c}>0\).
Since \(\rho_k>\rho^*\), the denominator in
\cref{eq:newtonupdate} is positive, and division of the two inequalities gives
 \cref{eq:bracket}.\nobreak\hfill
\end{proof}

\begin{theorem}[Finite termination and LP solution status of P-LPN]
\label{thm:finite}
Assume exact arithmetic and deterministic tie-breaking in Wolfe's nearest-point
subroutine.  Then every Wolfe call terminates.  Algorithm~\ref{alg:plpn}
returns either a primal--dual optimum for \((\mathrm P)\) as in
\cref{thm:polar-equivalence}(1), or a feasible recession direction certifying
that \((\mathrm P)\) is unbounded as in \cref{thm:polar-equivalence}(2).  It
performs at
most
\begin{equation}\label{eq:outer-bound}
  K(m+1,n):=1+\sum_{j=1}^{\min\{n+1,m+1\}}\binom{m+1}{j}
\end{equation}
completed outer iterations.
\end{theorem}

\begin{proof}
Wolfe's algorithm terminates finitely at every outer iteration
\citep{Wolfe1976}.  For each completed nonterminal iteration \(i\), let
\(C_i\subseteq\mathcal I\) index the terminal corral returned by Wolfe, and define
\[
 H_i:=\{\vct{q}:\vct{d}_i^T\vct{q}=h_i\},
 \qquad
 H_i^-:=\{\vct{q}:\vct{d}_i^T\vct{q}\leq h_i\}.
\]
Because \(\vct{z}_i\) is the affine minimizer over
\(\operatorname{aff}\{\vct{v}_\ell:\ell\in C_i\}\), \(\vct{d}_i\) is
orthogonal to every direction in that affine hull.  Hence each \(\ell\in C_i\)
satisfies
\(\vct{d}_i^T\vct{v}_\ell=\vct{d}_i^T\vct{z}_i=h_i\), so every generator
indexed by \(C_i\) lies in \(H_i\).  Projection optimality gives
\(\Q\subset H_i^-\), and
\[
 \vct{d}_i^T\vct{y}_i-h_i
 =\norm{\vct{y}_i-\vct{z}_i}^2>0.
\]
The update places \(\vct{y}_{i+1}\) on \(H_i\).  By
\cref{lem:update}, subsequent radii decrease strictly and
\(\vct{d}_i^T\vct{c}>0\); consequently,
\(\vct{y}_k\in H_i^-\) for every \(k>i\).

We now prove directly that a corral cannot repeat before termination.  Suppose
that \(C_j=C_k\) for two nonterminal iterations \(j<k\).  Since
\(\vct{z}_j\) and \(\vct{z}_k\) are convex combinations of the same
generators,
\[
 \vct{d}_j^T\vct{z}_k=h_j,
 \qquad
 \vct{d}_k^T\vct{z}_j=h_k.
\]
Because \(\vct{y}_k\in H_j^-\),
\[
 \vct{d}_j^T\vct{d}_k
 =\vct{d}_j^T(\vct{y}_k-\vct{z}_k)\leq0.
\]
Symmetry of the inner product gives
\(\vct{d}_k^T(\vct{y}_j-\vct{z}_j)\leq0\), and therefore
\(\vct{d}_k^T\vct{y}_j\leq h_k\).  On the other hand,
\(\rho_j>\rho_k\) and \(\vct{d}_k^T\vct{c}>0\), so
\[
 \vct{d}_k^T\vct{y}_j
 >\vct{d}_k^T\vct{y}_k
 =h_k+\norm{\vct{d}_k}^2
 >h_k,
\]
a contradiction.  Thus \(i\mapsto C_i\) is injective on the nonterminal outer
iterations.  This is the row-polar specialization of the support-order
argument in \citet[Theorem~3.11]{FujishigeEtAl2009}.

Every \(C_i\) indexes a nonempty affinely independent subset of the \(m+1\)
generators, so
\[
  1\leq |C_i|\leq \min\{n+1,m+1\}.
\]
There are therefore at most
\(\sum_{j=1}^{\min\{n+1,m+1\}}\binom{m+1}{j}=K(m+1,n)-1\)
nonterminal iterations.  An infinite run would require infinitely many
distinct such corrals, which is impossible.  Hence a terminal iteration occurs
after finitely many nonterminal iterations, giving the bound
\cref{eq:outer-bound}.  At that iteration, Algorithm~\ref{alg:plpn} returns
through either its finite-optimum or recession branch; \cref{thm:polar-equivalence}
 proves the stated LP conclusion.\nobreak\hfill
\end{proof}

In floating-point arithmetic, \textsc{Optimality Test Failed} means that the
candidate optimum was not confirmed, and \textsc{Recession Test Failed} means
the same for unboundedness.  Neither status asserts infeasibility: the origin is
strictly feasible for \((\mathrm P)\).

\subsection*{Explicit Rational-Data Outer-Iteration Bound}
\label{sec:rational-outer-bound}

The scalar distance viewpoint of
\citet[Remark~3.5]{FujishigeEtAl2009} suggests a polynomial Newton-step count
for rational-data LP-Newton.  The next theorem makes that observation explicit
for row-polar P-LPN.  It concerns completed outer iterations, not the Wolfe work
inside them.

Unit normalization can make a rational objective vector irrational.  For this
analysis, therefore, choose a nonzero rational vector
\(\boldsymbol{\gamma}\) parallel to \(\vct{c}\); one may take
\(\boldsymbol{\gamma}=\vct{c}^{\,0}\).  Parameterize the same ray by
\(\vct{y}(t):=t\boldsymbol{\gamma}\).  Assume
that the listed generators \(\vct{v}_0,\ldots,\vct{v}_m\),
\(\boldsymbol{\gamma}\), and the initial parameter \(t_0\) are rational, with
\[
 t_0>\max_{0\leq i\leq m}
 \frac{\boldsymbol{\gamma}^T\vct{v}_i}
      {\norm{\boldsymbol{\gamma}}^2}.
\]
This gives the same initial target as
\(\rho_0=t_0\norm{\boldsymbol{\gamma}}\).  For a nonempty closed convex set
\(\mathcal C\), let \(\operatorname{Proj}_{\mathcal C}(\vct{y})\) denote the
unique minimizer of \(\tfrac12\norm{\vct{q}-\vct{y}}^2\) over
\(\vct{q}\in\mathcal C\).  Define
\[
 t^*:=\max\{t\geq0:t\boldsymbol{\gamma}\in\Q\},\quad
 \vct{z}(t):=\operatorname{Proj}_{\Q}(t\boldsymbol{\gamma}),\quad
 \vct{d}(t):=t\boldsymbol{\gamma}-\vct{z}(t),\quad
 g(t):=\norm{\vct{d}(t)}.
\]

\begin{theorem}[Explicit rational-data outer-iteration bound]
\label{thm:rational-outer-bound}
Let \(\mathcal L\) be the total binary encoding length of
\(\vct{v}_0,\ldots,\vct{v}_m,\boldsymbol{\gamma}\), and \(t_0\).  If P-LPN is run in exact
arithmetic with exact Euclidean projections, it uses
\begin{equation}\label{eq:rational-outer-order}
 O(n^2\mathcal L)
\end{equation}
completed outer iterations.  A coarser bound stated only in input length is
\(O(\mathcal L^3)\).  Neither bound includes the inner Wolfe operations
performed within those iterations.
\end{theorem}

\begin{proof}
We first obtain a Newton-halving inequality, then describe the distance exactly
near the endpoint, and finally bound the rational constants that enter these
two arguments.

For \(t>t^*\), the point \(t\boldsymbol{\gamma}\) lies outside \(\Q\).  The squared-distance
function \(\tfrac12\operatorname{dist}(\vct{y},\Q)^2\) is differentiable with
gradient \(\vct{y}-\operatorname{Proj}_{\Q}(\vct{y})\).  Hence, for
\(F(t):=\tfrac12 g(t)^2\), the chain rule gives
\(F'(t)=\boldsymbol{\gamma}^T\vct{d}(t)\).  Since \(g(t)>0\), division by
\(g(t)\) gives
\begin{equation}\label{eq:distance-derivative-main}
 g'(t)=\frac{\boldsymbol{\gamma}^T\vct{d}(t)}{g(t)},
 \qquad 0<g'(t)\leq\norm{\boldsymbol{\gamma}}.
\end{equation}
The strict lower bound on \(g'(t)\) follows by applying the projection inequality to
\(\vct{q}^*:=t^*\boldsymbol{\gamma}\).  The supporting-hyperplane update is therefore Newton's
update for \(g(t)=0\):
\begin{equation}\label{eq:distance-newton-main}
 t_+:=t-\frac{g(t)}{g'(t)}
 =t-\frac{\norm{\vct{d}(t)}^2}{\boldsymbol{\gamma}^T\vct{d}(t)}
 =\frac{\vct{d}(t)^T\vct{z}(t)}{\vct{d}(t)^T\boldsymbol{\gamma}}.
\end{equation}
If \(t_+>t^*\), convexity of \(g\) between \(t_+\) and \(t\) gives
\begin{equation}\label{eq:distance-halving-main}
 \frac{g(t_+)}{g(t)}+\frac{g'(t_+)}{g'(t)}\leq1.
\end{equation}
Indeed, convexity gives
\(g(t)\geq g(t_+)+g'(t_+)(t-t_+)\), while
\cref{eq:distance-newton-main} gives
\(t-t_+=g(t)/g'(t)\); division by \(g(t)\) proves the display.  Therefore, at
least one of the two ratios in \cref{eq:distance-halving-main} is at most
one-half at every nonfinal update.

To see when the halving must stop, put
\(T:=\operatorname{cone}(\Q-\vct{q}^*)\), where
\(\operatorname{cone}\) denotes conic hull.  Let
\(\vct{w}:=\operatorname{Proj}_T(\boldsymbol{\gamma})\).  Define
\(\boldsymbol{\gamma}_{\perp}:=\boldsymbol{\gamma}-\vct{w}\) and
\(\delta:=\norm{\boldsymbol{\gamma}_{\perp}}\).  Here \(\delta>0\), because
\(\boldsymbol{\gamma}\in T\) would extend the objective ray beyond \(\vct{q}^*\).
If \(\vct{w}\neq\vct{0}\), conic Carath\'eodory gives an index set
\(J\subseteq\mathcal I\), with \(|J|\leq n\), linearly independent
directions \(\vct{r}_j:=\vct{v}_j-\vct{q}^*\), and positive coefficients
\(\alpha_j\) for \(j\in J\), such that
\(\vct{w}=\sum_{j\in J}\alpha_j\vct{r}_j\).  Let
\(R:=[\vct{r}_j]_{j\in J}\),
\(\vct{\alpha}:=(\alpha_j)_{j\in J}\),
\(\mu:=\sum_{j\in J}\alpha_j\), and
\(\bar\varepsilon:=\min\{1,(2\mu)^{-1}\}\).  If
\(\vct{w}=\vct{0}\), instead set \(\bar\varepsilon:=1\).
Projection onto the cone gives
\(\boldsymbol{\gamma}_{\perp}^T\vct{w}=0\) and
\(\boldsymbol{\gamma}_{\perp}^T\vct{r}\leq0\) for \(\vct{r}\in T\).  Hence, for
\(\vct{w}\neq\vct{0}\) and
\(0\leq\varepsilon\leq\bar\varepsilon\),
\(\vct{q}^*+\varepsilon\vct{w}
=(1-\varepsilon\mu)\vct{q}^*
+\varepsilon\sum_{j\in J}\alpha_j\vct{v}_j\in\Q\).
When \(\vct{w}=\vct{0}\), the same inclusion holds for
\(0\leq\varepsilon\leq\bar\varepsilon\) because
\(\vct{q}^*+\varepsilon\vct{w}=\vct{q}^*\).  Thus, in either case,
for any \(\vct{q}=\vct{q}^*+\vct{r}\in\Q\),
\[
 \bigl[\varepsilon\boldsymbol{\gamma}_{\perp}\bigr]^T
 \bigl[\vct{q}-(\vct{q}^*+\varepsilon\vct{w})\bigr]
 =\varepsilon\boldsymbol{\gamma}_{\perp}^T
   (\vct{r}-\varepsilon\vct{w})\leq0.
\]
This is the projection variational inequality, and therefore
\begin{equation}\label{eq:final-linear-main}
 \operatorname{Proj}_{\Q}((t^*+\varepsilon)\boldsymbol{\gamma})
 =\vct{q}^*+\varepsilon\vct{w},\qquad
 g(t^*+\varepsilon)=\varepsilon\delta.
\end{equation}
For arbitrary \(t>t^*\) and \(\vct{q}=\vct{q}^*+\vct{r}\in\Q\),
\[
 \boldsymbol{\gamma}_{\perp}^T
 \bigl(t\boldsymbol{\gamma}-\vct{q}\bigr)
 =(t-t^*)\delta^2-\boldsymbol{\gamma}_{\perp}^T\vct{r}
 \geq(t-t^*)\delta^2.
\]
Cauchy--Schwarz and minimization over \(\vct{q}\in\Q\) give
\(g(t)\geq\delta(t-t^*)\).  Since \(g\) is convex and
\cref{eq:final-linear-main} has slope \(\delta\) immediately to the right of
\(t^*\), also \(g'(t)\geq\delta\).  Consequently,
\(g(t)\leq\bar\varepsilon\delta\) implies
\(t-t^*\leq\bar\varepsilon\); the local linear formula then makes the next
Newton update exactly \(t^*\).

Equation~\eqref{eq:distance-halving-main} now bounds the number of nonfinal
updates by
\begin{equation}\label{eq:explicit-rational-outer-main}
 \left\lceil\log_2\max\left\{1,
 \frac{g(t_0)}{\bar\varepsilon\delta}\right\}\right\rceil
 +\left\lceil\log_2\max\left\{1,
 \frac{\norm{\boldsymbol{\gamma}}}{\delta}\right\}\right\rceil.
\end{equation}
At most two more projections reach and recognize the endpoint.

\smallskip
\noindent\emph{Encoding estimate.}
We now relate the constants in \cref{eq:explicit-rational-outer-main} to the
input encoding.  Write every rational coordinate of
\(\vct{v}_0,\ldots,\vct{v}_m\), \(\boldsymbol{\gamma}\), and \(t_0\) in lowest
terms, and let \(D\) be the product of their positive denominators.
Multiplication by \(D\) clears every input denominator.  Because
\(\log_2D\) is the sum of the logarithms of those denominators, \(D\) has
binary length at most \(\mathcal L\).  If \(M\vct{x}=\vct{b}\) is a nonsingular integer
system of order \(r\), and every entry of \(M\) and \(\vct{b}\) has absolute
value at most \(U\), then
\[
 1\leq |\det M|\leq r^{r/2}U^r.
\]
Hadamard's inequality gives the same upper bound for every Cramer numerator.
Thus Cramer's rule gives binary length
\(O(r(\log_2 U+\log_2 r))\) for every coordinate of its rational solution.
We use the elementary consequence that a nonzero rational number whose
numerator and denominator have binary length at most \(B\) has magnitude
between \(2^{-B}\) and \(2^B\).

The endpoint \(t^*\) is attained at an optimal basic solution of the rational LP
\[
 \begin{aligned}
 \operatorname*{maximize}_{\vct{\lambda},t}\quad &t\\[-2pt]
 \text{\rm subject to}\quad
 &\sum_{i=0}^m\lambda_i\vct{v}_i=t\boldsymbol{\gamma},\qquad
 \sum_{i=0}^m\lambda_i=1,\qquad
 \vct{\lambda}\geq\vct{0},\quad t\geq0.
 \end{aligned}
\]
Any optimal basic solution is determined by a nonsingular basis system of
order at most \(n+1\).  Multiplying this system by \(D\) gives an integer
coefficient matrix and right-hand side whose entries have magnitude
\(2^{O(\mathcal L)}\).  Cramer's rule therefore gives binary length
\(O(n(\mathcal L+\log(n+1)))\) for \(t^*\),
\(\vct{q}^*=t^*\boldsymbol{\gamma}\), and every coordinate of
\(\vct{v}_i-\vct{q}^*\).  Moreover, \(D\) together with the determinant of
this integer basis matrix gives a common denominator of the same
binary-length order for all coordinates of \(\boldsymbol{\gamma}\) and the
directions \(\vct{v}_i-\vct{q}^*\).

If \(\vct{w}\neq\vct{0}\), the positive coefficients in its conic
representation solve the nonsingular
Gram system
\((R^TR)\vct{\alpha}=R^T\boldsymbol{\gamma}\) of order at most \(n\).
After clearing a common denominator from \(R\) and
\(\boldsymbol{\gamma}\), each integer Gram entry has binary length
\(O(n(\mathcal L+\log(n+1)))\).  Applying the determinant estimate to this
system gives binary length \(O(n^2\mathcal L)\) for \(\vct{\alpha}\),
\(\vct{w}=R\vct{\alpha}\), and
\(\boldsymbol{\gamma}_{\perp}=\boldsymbol{\gamma}-\vct{w}\).
The coordinates of \(\vct{\alpha}\) share the Gram determinant as a
denominator, so \(\mu=\mathbf 1^T\vct{\alpha}\) has the same asymptotic
encoding length.  Since \(\boldsymbol{\gamma}_{\perp}\neq\vct{0}\), writing
its coordinates over a common denominator shows that at least one integer
numerator has magnitude at least one.  Therefore
\(\delta=\norm{\boldsymbol{\gamma}_{\perp}}\) and
\(\bar\varepsilon=\min\{1,(2\mu)^{-1}\}\) are bounded below by
\(2^{-O(n^2\mathcal L)}\).  If \(\vct{w}=\vct{0}\), then
\(\bar\varepsilon=1\) and \(\delta=\norm{\boldsymbol{\gamma}}\), so the same
lower bound follows directly from the nonzero rational input vector.
Finally, because \(\vct{0}\in\Q\),
\(g(t_0)\leq |t_0|\norm{\boldsymbol{\gamma}}\), and the input encoding bounds
both factors.  Consequently a universal constant \(C_{\rm enc}>0\) satisfies
\[
 \delta,\bar\varepsilon\geq2^{-C_{\rm enc}n^2\mathcal L},
 \qquad g(t_0),\norm{\boldsymbol{\gamma}}\leq2^{C_{\rm enc}n^2\mathcal L}.
\]
Each logarithm in \cref{eq:explicit-rational-outer-main} is therefore
\(O(n^2\mathcal L)\).  This proves \cref{eq:rational-outer-order}.  Since
\(n\leq\mathcal L\) for a nontrivial encoded instance, the coarser
 \(O(\mathcal L^3)\) bound follows.\nobreak\hfill
\end{proof}

For rational \(A^0,\vct{b}^{\,0},\bar{\vct{x}}^{\,0},\vct{c}^{\,0}\), the
normalized generators are rational and have encoding length polynomial in the
original input length, so the theorem also gives a polynomial outer-iteration
count in that input.  Suppose further that all generator coordinates and all
coordinates of the rational ray direction belong to a fixed finite set of
rational numbers.  If \(t_0\) is chosen as one plus the maximum in the
initial-target bound above, the same determinant argument gives
\(O(n^2\log(n+1))\) outer iterations, independent of numerical bit lengths.
This is a strongly polynomial bound on the number of outer iterations for that special class,
not a strongly polynomial algorithm.

The arbitrary-real combinatorial bound in \cref{eq:outer-bound} can be
exponential.  Moreover, no polynomial bound is known for Wolfe's inner
algorithm under the linear-optimization rule; the known exponential example
uses the different minimum-norm insertion rule
\citep{DeLoeraEtAl2020,FujishigeEtAl2025}.  Thus
\cref{thm:rational-outer-bound} does not prove weak or strong polynomiality of
the complete P-LPN method.

\section{Projection and Verification Choices}
\label{sec:validity-conditions}

Two suggestions in \citet[Section 5]{Fujishige2019} could in principle improve
the row-polar method: weighting distance along the objective ray and testing a
selected group of generators before the rest.  Ray weighting can produce a
larger update, but it changes the geometry, may worsen conditioning, and did
not improve runtime consistently.  Testing selected generators first changes
only their order; all remaining generators must still be checked.

We call that final test over the entire generator list \emph{global Wolfe
verification}.  Candidate-first ordering by itself was no faster.  We therefore
use Euclidean projection, repair the preceding corral, test that small set
first, and then complete global verification.  Appendix~\ref{app:projection-details}
reports the metric variant and the no-repair ordering control.

\section{CR-P-LPN: Algorithm and Analysis}
\label{sec:corralreuse}

Let \(V:=[\vct{0},\vct{a}_1,\ldots,\vct{a}_m]\) denote the generator matrix,
and let \(V_C\) collect the columns indexed by \(C\subseteq\mathcal I\).
CR-P-LPN solves the same projection problem as P-LPN; only Wolfe's starting
state changes.

\subsection{Corral Repair}

At a new target, CR-P-LPN repairs the preceding terminal corral before calling
Wolfe.  The result is an initial active set, not yet the new projection; Wolfe
must still verify optimality over the full hull.

Fix an outer iteration and abbreviate the preceding corral, its weights, and
the new target by
\[
 C:=C_k,\qquad
 \vct{\lambda}:=\vct{\lambda}^{\,k}_{C_k},\qquad
 \vct{y}^+:=\rho_{k+1}\vct{c}.
\]
Write \(\mathbf 1\) for an all-ones vector of the required dimension.  Repair
first computes the affine minimizer
\begin{equation}\label{eq:corral-repair}
 \vct{\lambda}^{\mathrm{aff}}
 :=\argmin_{\mathbf 1^T\vct{\mu}=1}
 \tfrac12\norm{V_C\vct{\mu}-\vct{y}^+}^2.
\end{equation}
If any component of \(\vct{\lambda}^{\mathrm{aff}}\) is nonpositive, set
\[
 \theta:=\min_{i\in C:\lambda_i^{\mathrm{aff}}\leq0}
 \frac{\lambda_i}{\lambda_i-\lambda_i^{\mathrm{aff}}},
 \qquad
 \vct{\lambda}\leftarrow
 (1-\theta)\vct{\lambda}+\theta\vct{\lambda}^{\mathrm{aff}}.
\]
This step preserves nonnegative weights and brings at least one of them to zero.
After deleting the zero-weight generators, we solve \cref{eq:corral-repair}
again.  At most \(|C|-1\) passes leave a nonempty corral with positive weights.
We denote the resulting pair by
\(\mathrm{Repair}(\vct{y}^+;C_k,\vct{\lambda}^{\,k}_{C_k})\).

\subsection{Implicit CR-P-LPN}

The initial projection starts cold.  An integer \(B\geq1\) permits warm starts
only at indices \(k=1,\ldots,B-1\); it does not limit the total outer
iterations.  Larger \(B\) creates more opportunities for reuse but can also add
repair and verification work.  The correctness result below holds for every
finite \(B\).  The flag \(\mathrm{coldOnly}\) records when all remaining
projections must start cold.

In \cref{alg:corralreuse}, \(\mathrm{Wolfe}\) is the cold call from
\cref{alg:plpn}, and \(\mathrm{Repair}\) is the procedure above.
\(\mathrm{WarmWolfe}\) begins from the repaired pair, tests the candidate set
\(\mathcal J_k\subseteq\mathcal I\), and then checks every remaining generator.
\(\mathrm{StatusUpdate}\) carries out the stopping and radius-update branches
of P-LPN.  It returns a status \(\sigma_k\), the associated result
\(\mathcal R_k\), and the next radius; details follow the algorithm.

\begin{algorithm}[H]
\caption{Implicit CR-P-LPN}
\label{alg:corralreuse}
\footnotesize
\begin{algorithmic}[1]
\Require normalized problem \((\mathrm P)\), index set \(\mathcal I\), and
generator access through \(A\)
\Require \(\rho_0>\max_{i\in\mathcal I}
\vct{c}^T\vct{v}_i\), reuse limit \(B\in\mathbb N\)
\State \(\mathrm{coldOnly}\gets0\), \quad
\((\vct{d}_{-1},h_{-1})\gets(\vct{0},0)\)
\For{\(k=0,1,\ldots\)}
  \State \(\vct{y}_k\gets\rho_k\vct{c}\)
  \If{\(k=0\) or \(k\geq B\) or \(\mathrm{coldOnly}=1\)}
    \State \((\vct{z}_k,C_k,\vct{\lambda}^{\,k}_{C_k})
    \gets\mathrm{Wolfe}(\vct{y}_k;\mathcal I)\)
  \Else
    \State \((\widetilde C_k,\widetilde{\vct{\lambda}}_k)
    \gets\mathrm{Repair}(\vct{y}_k;C_{k-1},
    \vct{\lambda}^{\,k-1}_{C_{k-1}})\)
    \State \((\vct{z}_k,C_k,\vct{\lambda}^{\,k}_{C_k},\mathrm{ok})
    \gets\mathrm{WarmWolfe}(\vct{y}_k;\widetilde C_k,
    \widetilde{\vct{\lambda}}_k,\mathcal J_k)\)
    \If{\(\mathrm{ok}=0\)}
      \State \(\mathrm{coldOnly}\gets1\), \quad
      \((\vct{z}_k,C_k,\vct{\lambda}^{\,k}_{C_k})
      \gets\mathrm{Wolfe}(\vct{y}_k;\mathcal I)\)
    \EndIf
  \EndIf
  \State \(\vct{d}_k\gets\vct{y}_k-\vct{z}_k\), \quad
  \(h_k\gets\vct{d}_k^T\vct{z}_k\)
  \State \((\sigma_k,\mathcal R_k,\rho_{k+1})
  \gets\mathrm{StatusUpdate}(\rho_k,\vct{z}_k,C_k,
  \vct{\lambda}^{\,k}_{C_k};\vct{d}_{k-1},h_{k-1},\vct{d}_k,h_k)\)
  \If{\(\sigma_k\neq\mathrm{Continue}\)}
    \State \Return \((\sigma_k,\mathcal R_k)\)
  \EndIf
\EndFor
\end{algorithmic}
\end{algorithm}

\Cref{fig:corral-repair-mechanism} shows why repair and verification are
separate operations.

\begin{figure}[H]
\centering
\resizebox{0.98\textwidth}{!}{\centering
\begin{tikzpicture}[
  x=1cm,y=1cm,>=Stealth,
  every node/.style={font=\scriptsize},
  point/.style={circle,fill=black,inner sep=1.35pt},
  flow/.style={draw=black!55,fill=black!3,rounded corners=1pt,
    align=center,text width=1.75cm,minimum height=.62cm,inner sep=2.5pt}
]
  \node[font=\scriptsize\bfseries] at (2.0,2.75) {(a) Affine repair};
  \coordinate (A) at (0,0);
  \coordinate (B) at (4.0,0);
  \coordinate (C) at (1.15,2.0);
  \coordinate (L) at (1.30,.66);
  \coordinate (Lp) at ($(B)!0.25!(C)$);
  \coordinate (G) at ($(L)!1.55!(Lp)$);
  \fill[black!6] (A)--(B)--(C)--cycle;
  \draw[black!55,thick] (A)--(B)--(C)--cycle;
  \node[black!65,align=center] at (1.38,1.45)
    {active convex\\hull};
  \draw[blue!65!black,very thick,-{Stealth[length=2.2mm]}] (L)--(Lp)
    node[pos=.40,above=4pt,text=blue!65!black] {minor step};
  \draw[red!70!black,dashed,thick,-{Stealth[length=2.0mm]}] (Lp)--(G);
  \node[point,blue!65!black,label={[text=blue!65!black]below left:
    \(V_C\vct{\lambda}\)}] at (L) {};
  \node[point,green!45!black,label={[text=green!45!black]below:
    \(V_C\vct{\lambda}^{+}\)}] at (Lp) {};
  \node[point,red!70!black,label={[text=red!70!black]right:
    \(V_C\vct{\lambda}^{\rm aff}\)}] at (G) {};
  \node[black!70,align=center] (zero) at (3.85,1.35)
    {first zero\\weight};
  \draw[black!70,-{Stealth[length=1.8mm]}] (zero.south west)--(Lp);
  \begin{scope}[xshift=5.45cm]
    \node[font=\scriptsize\bfseries] at (4.2,2.75)
      {(b) Verified warm projection};
    \node[flow] (old) at (0.75,2.05) {preceding\\terminal corral};
    \node[flow] (repair) at (3.15,2.05) {repair at\\new target};
    \node[flow] (warm) at (5.65,2.05) {warm Wolfe};
    \node[flow] (exact) at (8.15,2.05) {projection\\\(\vct{z}_k\)};
    \node[flow] (cold) at (5.65,.35) {cold Wolfe};
    \draw[->,thick,black!65] (old)--(repair);
    \draw[->,thick,black!65] (repair)--(warm);
    \draw[->,thick,black!65] (warm)--(exact);
    \draw[->,thick,black!65] (warm)--
      node[right,align=left]{unverified\\stop}(cold);
    \draw[->,thick,black!65] (cold.east)-|
      node[pos=.28,below]{recompute}(exact.south);
  \end{scope}
\end{tikzpicture}}
\caption{Corral repair changes Wolfe's initialization, not the projection problem.
\label{fig:corral-repair-mechanism}}
\begin{minipage}{0.98\textwidth}
\scriptsize Panel (a) shows the case \(\theta<1\): repair moves toward the
  affine minimizer only until one weight first reaches zero, then deletes that
  generator and resolves.  In panel (b), warm Wolfe continues after any violating
  generator.  It returns only after the target itself is the projection or after
  global Wolfe verification;
  if it stops before verification, cold Wolfe recomputes the projection.
\end{minipage}
\end{figure}
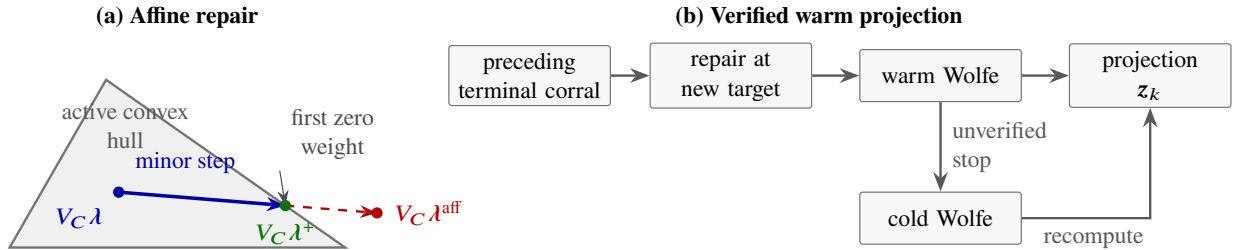

\smallskip
\noindent\textbf{Interface details.}
Define the global Wolfe violation
\(\Delta(\vct{y},\vct{z})
:=\max_{j\in\mathcal I}(\vct{y}-\vct{z})^T(\vct{v}_j-\vct{z})\).
For \(\vct{z}\in\Q\), \(\Delta(\vct{y},\vct{z})\leq0\) if and only if
\(\vct{z}\) is the projection of \(\vct{y}\) onto \(\Q\).
The set \(\mathcal J_k\) contains the repaired corral.  From this corral and its
weights, \(\mathrm{WarmWolfe}\) returns
\((\vct{z}_k,C_k,\vct{\lambda}^{\,k}_{C_k},\mathrm{ok})\).
Here \(\mathrm{ok}=0\) only when a numerical or resource limit stops the call
before verification.  Finding an ordinary violating generator simply makes
Wolfe continue.
If \(\vct{y}_k=\vct{z}_k\), the target itself is the projection and no global
test is needed.

The status output \(\sigma_k\) belongs to
\(\{\mathrm{Continue},\mathrm{Optimal},\mathrm{OptTestFailed},
\mathrm{Unbounded},\mathrm{RecTestFailed}\}\).
The result is \(\mathcal R_k:=\emptyset\) for \(\mathrm{Continue}\),
\(\mathcal R_k:=(\vct{x}^*,\vct{u}^*)\) for \(\mathrm{Optimal}\) or
\(\mathrm{OptTestFailed}\), and \(\mathcal R_k:=\vct{r}\) for
\(\mathrm{Unbounded}\) or \(\mathrm{RecTestFailed}\).
A terminal call returns \(\rho_k\) in the unused radius slot.  The dummy pair
\((\vct{d}_{-1},h_{-1}):=(\vct{0},0)\) is never read because
\(\vct{y}_0\notin\Q\).  Returned weights are indexed by their corral;
\(\mathrm{StatusUpdate}\) extends them by zeros when recovery needs all row
weights.

\smallskip
\noindent\textbf{Implicit generator operations.}
The implementation works directly with \(A\); the matrix \(V\) is notation,
not stored data.
For
\(\vct{\lambda}=(\lambda_0,\lambda_1,\ldots,\lambda_m)^T\), write
\(\vct{\lambda}_{1:m}:=(\lambda_1,\ldots,\lambda_m)^T\).  Then, for
\(\vct{d}\in\R^n\),
\begin{equation}\label{eq:implicit-polar}
 V\vct{\lambda}
 =A^T\vct{\lambda}_{1:m},
 \qquad
 V^T\vct{d}
 =\begin{bmatrix}0\\
 A\vct{d}
 \end{bmatrix}.
\end{equation}
Thus \(A\) and \(A^T\) supply all generator scores and combinations without a
separate generator matrix.

Let \(\operatorname{nnz}(A)\) denote the number of stored nonzeros in \(A\).
By \cref{eq:implicit-polar}, a full row-polar scan and a row-polar convex
combination each cost \(O(\operatorname{nnz}(A)+m)\).  The score, weight, and
index vectors have length \(m+1\), while an affinely independent corral has at
most \(n+1\) generators.  Its columns and Karush--Kuhn--Tucker (KKT) matrix
therefore need \(O(n^2)\) storage, and all remaining work vectors have length
\(n\).  Thus the extra storage beyond \(A\) is \(O(m+n^2)\), with no
separate \(n\)-by-\((m+1)\) generator matrix.

\FloatBarrier

\begin{theorem}[Solution status and common outer-target sequence of CR-P-LPN]
\label{thm:corralreuse}
Let \(B<\infty\), assume exact arithmetic and deterministic tie-breaking in
all Wolfe calls, and require global Wolfe verification before accepting a
projection of any
\(\vct{y}_k\notin\Q\).  Then CR-P-LPN terminates with the same solution status
for \((\mathrm P)\)---optimal or unbounded---and generates the same outer-target
sequence as P-LPN.  If \(N_{\rm P}\) and
\(N_{\rm CR}\) denote their completed outer-iteration counts, then
\begin{equation}\label{eq:corralreuse-bound}
 N_{\rm CR}=N_{\rm P}\leq K(m+1,n).
\end{equation}
Under the rational assumptions of \cref{thm:rational-outer-bound}, this common
count is \(O(n^2\mathcal L)\).
\end{theorem}

\begin{proof}
Use superscripts \({\rm P}\) and \({\rm CR}\) for the quantities generated by
P-LPN and CR-P-LPN.  Both methods start from the same radius, so
\(\vct{y}_0^{\rm P}=\vct{y}_0^{\rm CR}=\rho_0\vct{c}\); they also use the same
dummy preceding support data at \(k=0\).

Suppose at iteration \(k\) that the targets and preceding support data agree.
P-LPN computes \(\operatorname{Proj}_{\Q}(\vct{y}_k)\).  Repair changes only
Wolfe's initial active set.  If warm Wolfe returns the target itself or passes
global Wolfe verification, it has computed this same unique projection;
otherwise cold Wolfe recomputes it before any radius update.
Consequently,
\[
 \vct{z}_k^{\rm CR}
 =\operatorname{Proj}_{\Q}(\vct{y}_k)
 =\vct{z}_k^{\rm P},
 \qquad
 (\vct{d}_k^{\rm CR},h_k^{\rm CR})
 =(\vct{d}_k^{\rm P},h_k^{\rm P}).
\]
The two methods therefore take the same stopping branch.  If that branch is
nonterminal, \cref{eq:newtonupdate} gives
\(\rho_{k+1}^{\rm CR}=\rho_{k+1}^{\rm P}\), so their next targets and preceding
support data again agree.  Induction proves the common outer-target sequence,
the common solution status, and \(N_{\rm CR}=N_{\rm P}\).

The finite combinatorial bound and validity of the terminal status follow from
\cref{thm:finite}; under the rational assumptions, the sharper bound follows
from \cref{thm:rational-outer-bound}.  Work performed by an unsuccessful warm
 attempt is included in the operation counts below.\nobreak\hfill
\end{proof}

\begin{proposition}[Arithmetic-operation and storage bounds]
\label{prop:output-work}
Let \(G\) be the number of full-row passes used either to find a violating
generator or to verify Wolfe's condition, \(H\) the number of affine corral
solves, \(N_{\rm out}\) the number of completed outer iterations, and
\(r_h\leq n+1\) the number
of generators in the \(h\)-th affine solve.  Excluding the initial input read, the implicit
row-polar implementation uses
\begin{equation}\label{eq:output-work}
 O\left((G+1)(\operatorname{nnz}(A)+m)
 +\sum_{h=1}^{H}(nr_h^2+r_h^3)+N_{\rm out}n+n^3\right)
\end{equation}
arithmetic operations and
\(O(\operatorname{nnz}(A)+m+n^2)\) storage.  Full-row passes and affine corral
solves performed before any unverified warm exit are included in \(G\) and
\(H\), respectively.
\end{proposition}

\begin{proof}
\Cref{eq:implicit-polar} gives the full-row verification cost.  For a corral with \(r_h\)
generators, forming its Gram matrix costs \(O(nr_h^2)\), and solving or
estimating the condition of its KKT system costs \(O(r_h^3)\).  Ray updates
cost \(O(n)\).  Products with \(A\) and \(A^T\) used in verification are
included in the first term of \cref{eq:output-work}; optional full-rank support
polishing costs at most \(O(n^3)\).  The storage statement follows from the
 implicit-product accounting after \cref{eq:implicit-polar}.\nobreak\hfill
\end{proof}

If \(G,H,N_{\rm out}\) stay bounded over a family, \(r_h\leq n+1\) reduces
\cref{eq:output-work} to \(O(\operatorname{nnz}(A)+m+n^3)\), including one pass
over the stored matrix.  This conditional simplification does not bound the
three counts; it is not a worst-case polynomial result or a comparison theorem
for HiGHS.

\subsection{Extension: Reoptimization Across Objectives}
\label{sec:objective-reuse}

The same idea can transfer state across objectives.  For an integer \(T\geq1\),
consider the finite sequence
\begin{equation}\label{eq:objective-sequence}
  (\mathrm P_t)\qquad
  \max\{(\vct{c}^{\,t})^T\vct{x}:A\vct{x}\leq\mathbf 1\},
  \qquad \norm{\vct{c}^{\,t}}=1,\quad t=1,\ldots,T.
\end{equation}
All problems share \(A\) and \(\Q\); only the unit objective-ray direction
\(\vct{c}^{\,t}\) changes.
RO-CR-P-LPN solves the first LP cold and offers its final corral to the next
objective, repeating this process through the sequence.  Each transfer still
requires repair and global verification.  This extension is motivated by
objective-coefficient sensitivity; Appendix~\ref{app:objective-transfer} gives the
rule; both implementations are included in the versioned computational
archive described at the end of this report.

\begin{corollary}[Finite-sequence correctness]
\label{cor:objective-reuse}
Under the assumptions of \cref{thm:corralreuse}, RO-CR-P-LPN returns the same
finite optimum or unboundedness conclusion for every \((\mathrm P_t)\) as an unrestricted cold solve and
terminates for finite \(T\).  Shared preparation costs
\(O(\operatorname{nnz}(A)+m+n)\) operations and storage.  The remaining work
is the sum of \cref{eq:output-work} over the objectives, including work before
unverified warm exits, and peak storage is
\(O(\operatorname{nnz}(A)+m+n^2)\).
\end{corollary}

\begin{proof}
Transfer changes only Wolfe's initial active set.  The global test or cold
recomputation therefore gives the same projection as a cold solve, so
\cref{thm:corralreuse} applies to each objective.  Summing its work and using
 the implicit-product accounting and \cref{prop:output-work} gives the bounds.\nobreak\hfill
\end{proof}

Changing \(A\) changes \(\Q\) and is outside this result.

\FloatBarrier
\section{Numerical Experiments}\label{sec:numerics}

The experiments separate the projection algorithms from the cost of preparing
the row-polar problem.  Phase-II tests start from the same normalized instance
and compare P-LPN with CR-P-LPN without initialization cost.  End-to-end tests
include initialization through verification.  Component and boundary tests
then examine when repair saves work and when it adds work.

\subsection{Questions and Protocol}

The main experiments address three questions.  How do P-LPN and CR-P-LPN
compare after a common initialization and in complete solves, including against
cold HiGHS dual simplex and IPM with crossover?  Does corral repair help beyond
changing the verification order or projection metric?  Which problem features
make the saved Wolfe work exceed the cost of repair?  A smaller, secondary
experiment asks whether a terminal corral can also help when the constraints
stay fixed and only the objective changes.

\emph{Implementation checks.}
We compare times only for outcomes that pass the original-scale tests.  A
finite result must satisfy primal and dual feasibility, nonnegativity of the
multipliers, the primal--dual gap, and radial agreement.  An unbounded result
must include a normalized feasible recession direction.  Known solutions and
independent solvers provide further checks where available.  Both
implementations return the expected status on the retained examples.  Appendix~\ref{app:verification} defines the residuals, tolerances, and optional multiplier
polishing.

\emph{Comparisons.}
Cold-start P-LPN is the primary baseline.  Component variants separate repair
from the other design choices.  An author-implemented primal-tableau simplex
provides an independent code path, while cold HiGHS dual simplex and IPM with
crossover are mature general-purpose references \citep{HuangfuHall2018}.
Phase-II comparisons begin after common row-polar initialization; end-to-end
comparisons start all methods from the same original arrays.  We do not add
MATLAB R2025b \texttt{linprog} because its default algorithm is HiGHS-based.

\emph{Instances.}
For the component study, 180 single-LP instances with known exposed polar faces cover regular,
diagonally conditioned, and degenerate cases in dimensions 10--100 with at
most 700 rows.  Three prespecified many-row instances have
\((n,m)=(16,5000),(32,25000),(64,100000)\), known exposed optima, and mostly
four to six nonzeros per row.  They test costly scans when \(m\gg n\), not a
general LP class.  Boundary tests include small status examples, Klee--Minty
dimensions three through eight, and a dimension-50 stress case
\citep{KleeMinty1972}.

The application-derived study contains nine box-regularized minimax
regressions, seven frozen tail-margin models, and three long-only maximin
portfolios: 19 LPs with 722--79,405 rows and 6--60 columns
\citep{UCIRepository2026,French2026Data}.  Fixed screening of 91 Netlib models
retains ten for a separate comparison \citep{Gay1985}.  The Netlib subset is a
boundary check rather than a benchmark of the full collection, and the
application-derived tests do not assess predictive quality.  The
secondary transfer study uses 24 twelve-objective sequences at four sizes,
three relevant-generator overlap patterns, and two seeds.  A separate broader
grid has 108 five-objective sequences, and a fixed-constraint 49-industry
sequence provides a further boundary test.
Appendices~\ref{app:comparators}--\ref{app:numerical-details} give comparator settings, verification rules,
and additional numerical results.  A versioned computational archive contains
the data transformations, Netlib screening rules, instance generators, control
definitions, provenance, and checksums; the Data and Code Availability statement
identifies the archive and explains how to obtain it.

\emph{Implementation and timing.}
For numerical stability, a positive diagonal transformation
\(\vct{x}^{\,0}=\bar{\vct{x}}^{\,0}+D\vct{\xi}\) replaces each row by
\((\vct{a}_i^{\,0})^T D/s_i^{\,0}\) and the objective by
\(D\vct{c}^{\,0}\); displacements and
recession directions are mapped back by \(D\).  This equivalent scaling
stabilizes Klee--Minty coefficients of widely different magnitudes.  Disjoint
development sets fix the ray weight at \(0.1\) for the Section~5 metric control
and the candidate-set fraction at \(0.5\); default CR-P-LPN uses the Euclidean
metric.
We use \(B=4\) in the reported experiments; application-derived instances
were not used for tuning.

Every timed run is retained.  Each study uses 3--10 randomized-order
repetitions after one unrecorded warm-up, disjoint seeds, and one solver and one
linear-algebra thread.  Phase-II time begins with the initialized row-polar
problem and includes solver setup, generator scans, unsuccessful warm work, and
any cold recomputation.  End-to-end time begins with the original arrays.  For
P-LPN and CR-P-LPN it includes construction of the artificial problem, Phase I,
any affine-hull reduction and repeated Phase-I call, normalization, Phase II,
inverse transformation, and verification.  For HiGHS it includes model
construction, presolve, optimization, any crossover, solution retrieval, and
verification.  File reading is excluded, and complete times are measured
directly rather than assembled from separate medians.

Application-derived checks use \(10^{-7}\) for normalized residuals and HiGHS
objective differences.  Other tolerances are \(10^{-6}\) for feasibility and
\(10^{-5}\) for relative gap and radial agreement.  These are floating-point
acceptance tests, not exact refinement \citep{GleixnerEtAl2016,EiflerEtAl2025}.
Results use Julia 1.8.3, MATLAB R2025b, and HiGHS 1.14.0 on an Intel Core
i9-9980HK MacBook Pro.  The archive records source hashes, repetitions, thread
settings, and implemented parameter settings.

\FloatBarrier

\subsection{Many-Row Runtime Comparison}
\label{sec:many-row}

The three prespecified many-row instances give the clearest test of expensive
full-row work.  All 120
Phase-II and 90 end-to-end runs return the prespecified optimum and pass the
original-scale tests.  Julia and MATLAB operation counts agree for every
Phase-II instance--method pair.

Phase-II P-LPN/CR-P-LPN geometric-mean time ratios are 2.03
[1.74,2.25] in Julia and 1.46
[1.33,1.56] in MATLAB.  Major-step counts fall from 70 to 25, 123 to 46,
and 215 to 98; full-row scans fall from 72 to 23, 125 to 40, and 217 to 77; and
affine solves fall from 100 to 37, 164 to 63, and 263 to 135.  These reductions
occur in exactly the operations that corral reuse is intended to avoid,
consistent with \cref{prop:output-work}.

\begin{table}[H]
\caption{End-to-end many-row comparison (mostly four to six nonzeros per row).
P/CR is P-LPN total time divided by CR-P-LPN total time; H/CR is the faster cold
HiGHS total time divided by CR-P-LPN total time.
\label{tab:many-row-time}}
\centering
\small\centering\begin{tabular}{lrrrrrrr}
\toprule
Language & $n$ & $m$ & P total & CR total & P/CR & H-fast & H/CR \\
\midrule
Julia & 16 & 5,000 & 0.0101 & 0.0074 & 1.31 & 0.0395 DS & 5.36 \\
Julia & 32 & 25,000 & 0.0926 & 0.0657 & 1.41 & 0.3016 DS & 4.59 \\
Julia & 64 & 100,000 & 1.8109 & 1.1394 & 1.58 & 1.9564 IPM & 1.72 \\
MATLAB & 16 & 5,000 & 0.0231 & 0.0173 & 1.33 & 0.0395 DS & 2.28 \\
MATLAB & 32 & 25,000 & 0.1572 & 0.0998 & 1.49 & 0.3016 DS & 3.02 \\
MATLAB & 64 & 100,000 & 1.8949 & 1.0332 & 1.83 & 1.9564 IPM & 1.89 \\
\bottomrule
\end{tabular}\ \null\par\smallskip
\begin{minipage}{0.98\textwidth}
\scriptsize Entries are medians of five randomized-order, single-thread
runs after an unrecorded run.  DS and IPM identify the faster cold HiGHS mode.
The P/CR entry is the median within-repeat ratio; H/CR uses the displayed
medians.  Complete totals are timed directly.  The saved instances already
have \(\vct{b}^{\,0}=\mathbf 1\) and the visible strict point
\(\vct{x}^{\,0}=\vct{0}\), so
forcing artificial Phase I is conservative for these files.  Appendix~\ref{app:numerical-details},
Table~\ref{tab:supp-many-row}, reports the Phase-I and Phase-II timing components.
Cross-language times are not ranked.
\end{minipage}
\end{table}

End-to-end P-LPN/CR-P-LPN geometric-mean ratios
are \EToEJuliaPCRGM{} in Julia and \EToEMatlabPCRGM{} in MATLAB.  CR-P-LPN is
faster for all three instances in both languages.  The geometric-mean ratios of
the faster cold HiGHS time to the CR-P-LPN total time are
\EToEJuliaHighsCRGM{} in Julia and \EToEMatlabHighsCRGM{} in MATLAB.  P-LPN is
also faster than the faster cold HiGHS mode for all three instances in both
languages.  On the
100,000-row instance, however, the corresponding HiGHS/P-LPN ratios narrow to
1.08 and 1.03.  The many-row advantage therefore survives the forced Phase-I
cost, although P-LPN's margin over HiGHS is small on the largest instance.
Three prespecified instances cannot establish a general solver ranking.

\FloatBarrier
\subsection{Component Evidence for Corral Repair}
\label{sec:ablation}

The component tests separate terminal-corral reuse from the other choices in
Section~5.  The ``Active-face scan'' control changes verification order but
does not repair a corral; CR-P-LPN does both.

\begin{table}[H]
\caption{Evaluation-set corral-repair component comparison.  Speed is paired P-LPN time
divided by method time; a paired win means that the method is faster on that
instance, and CI denotes confidence interval.
\label{tab:adaptive-ablation}}
\centering
\small\centering
\setlength{\tabcolsep}{4pt}
\begin{tabular}{@{}llrrrr@{}}
\toprule
Language & Method & Verified & Geom. mean time & Median speed (95\% CI) & Paired wins \\
\midrule
Julia & P-LPN & 180/180 & 0.005661 & 1.00 [1.00,1.00] & 0 \\
Julia & Corral reuse & 180/180 & 0.003396 & 1.66 [1.56,1.77] & 176 \\
Julia & Active-face scan & 180/180 & 0.005818 & 0.98 [0.97,1.00] & 73 \\
Julia & CR-P-LPN & 180/180 & 0.003433 & 1.62 [1.55,1.69] & 178 \\
\midrule
MATLAB & P-LPN & 180/180 & 0.009079 & 1.00 [1.00,1.00] & 0 \\
MATLAB & Corral reuse & 180/180 & 0.005566 & 1.61 [1.52,1.74] & 177 \\
MATLAB & Active-face scan & 180/180 & 0.009464 & 0.97 [0.95,0.98] & 56 \\
MATLAB & CR-P-LPN & 180/180 & 0.005839 & 1.55 [1.46,1.66] & 172 \\
\bottomrule
\end{tabular}\ \null\par\smallskip
\begin{minipage}{0.98\textwidth}
\scriptsize Times are seconds and medians of three post-warm-up runs.
Intervals use 10,000 paired bootstrap resamples.
\end{minipage}
\end{table}

All 4,320 runs satisfy the applicable verification tests.  Pure corral reuse gives
median paired speedups of 1.66 in Julia and 1.61 in MATLAB.  CR-P-LPN gives
1.62 and 1.55, winning 178/180 and
172/180 instance medians; active-face scanning alone gives 0.98 and 0.97.

The Spearman association between each instance's P-LPN/CR-P-LPN time ratio and
its affine-solve reduction is
\JuliaAffineAssociation{} [\JuliaAffineAssociationLower{},
\JuliaAffineAssociationUpper{}] in Julia and
\MatlabAffineAssociation{} [\MatlabAffineAssociationLower{},
\MatlabAffineAssociationUpper{}] in MATLAB; the corresponding associations for
calls to Wolfe's linear-optimization step are \JuliaScanAssociation{}
[\JuliaScanAssociationLower{},\JuliaScanAssociationUpper{}] and
\MatlabScanAssociation{} [\MatlabScanAssociationLower{},
\MatlabScanAssociationUpper{}].  The 95\% intervals use 10,000 bootstrap
resamples.  These are associations, not causal estimates.

Component comparisons are consistent with terminal-corral reuse as the useful
change.  Candidate-first ordering and metric weighting alone show no consistent
gain, whereas repair reduces Wolfe steps, scans, and affine solves.

\FloatBarrier
\subsection{Conditions and Performance Boundaries}
\label{sec:real-endpoints}

The application-derived end-to-end study has 570 timing records, 190 each
for Julia, MATLAB, and HiGHS.  Every run returns an optimum and passes the
original-scale tests, with maximum relative objective error
\(3.43\times10^{-8}\).  \Cref{tab:application-end-to-end} summarizes the
19 workload medians.

\begin{table}[H]
\caption{End-to-end comparison on 19 application-derived LPs.
\label{tab:application-end-to-end}}
\centering
\scriptsize\centering\begin{tabular}{@{}lcc@{}}
\toprule
& Julia & MATLAB \\
Time ratio & GM [95\% CI], count $>1$ & GM [95\% CI], count $>1$ \\
\midrule
P-LPN/CR-P-LPN & 1.25 [1.17,1.34], 18/19 & 1.09 [1.02,1.16], 12/19 \\
H-fast/P-LPN & 3.93 [2.85,5.30], 18/19 & 2.82 [2.19,3.53], 18/19 \\
H-fast/CR-P-LPN & 4.93 [3.44,6.74], 18/19 & 3.06 [2.35,3.86], 18/19 \\
\bottomrule
\end{tabular}\ \null\par\smallskip
\begin{minipage}{0.98\textwidth}
\scriptsize Entries are geometric means (GM) across workload medians, with
95\% bootstrap confidence intervals (CI) and counts of ratios above one.  Each
workload has five randomized-order repetitions.  H-fast is the smaller cold
HiGHS median from dual simplex and IPM with crossover.  Ratios above one favor
the denominator.
\end{minipage}
\end{table}

Dual simplex is the faster HiGHS mode for all 19 workloads.  The common
exception is the Spambase tail-margin LP, where the H-fast/CR-P-LPN ratios are
0.59 and 0.65 in Julia and MATLAB, respectively.  Thus the comparison is
favorable but does not establish universal solver dominance.
Every row-polar Phase I succeeds in one call, so none of the 19
application-derived LPs requires affine-hull reduction.  The end-to-end
protocol includes that branch when required; Appendix~\ref{app:initialization} gives its algorithm,
and Appendix~\ref{app:numerical-details} gives the designed validation tests.

All runs in the separate \RealEndpointPrimaryRuns-run Phase-II study also
satisfy the verification tests.  Its P/CR geometric means are \JuliaRealCRSpeed\ and
\MatlabRealCRSpeed{}.  CR-P-LPN is faster on 13 Julia and 12 MATLAB workload
medians, but both row-polar methods are slower than HiGHS on the largest
regression.  The gain is less consistent than on the many-row tests.

\label{sec:klee50}

Against the author-simplex implementation on the 180 synthetic problems,
CR-P-LPN time ratios are 0.55 [0.46,0.66] in Julia and 0.63 [0.46,1.05] in
MATLAB; every pair passes the applicable tests.  On ten retained Netlib models, both methods pass all
ten in each language, but CR-P-LPN takes 3.44 and 3.34 times the simplex time;
simplex wins 10/10 Julia and 9/10 MATLAB medians.  These are implementation and
boundary checks, not corral-reuse comparisons.

For the dimension-50 Klee--Minty bounded-zonotope problem,
\citet{FujishigeEtAl2009} report two Newton steps, 12 generated points, and
0.0113 seconds, with optimizer \((0,\ldots,0,100^{49})^T\) and value \(10^{98}\).
Our bounded runs reproduce two steps; our 11 points reflect a different
starting convention, and cross-machine times are not compared.  In scaled
row-polar form, all 160 runs pass the verification tests with optimizer \(\vct{e}_{50}\) and value one.
P-LPN takes five outer iterations, the metric variant four, and CR-P-LPN invokes a
cold restart; CR-P-LPN/P-LPN time ratios are 1.06 in Julia and 1.34 in MATLAB.

Many rows alone do not explain the gain.  Repair helps when checking all
generators is costly \emph{and} successive projections retain useful corral
generators.  If the preceding corral is uninformative, repair, verification,
and cold recomputation can cost more than the Wolfe work they save.  The
application-derived, Klee--Minty, Spambase, and Netlib results therefore rule
out a general dominance claim.

\subsection{Secondary Test: Transfer Across Objectives}
\label{sec:repeated-objective-results}

All \ROPrimaryObjectiveSolves{} row-polar solves satisfy the verification tests.  On the 18 sequences
with at least 5,000 rows, separate CR-P-LPN takes \ROJuliaCRSpeedupMany{} and
\ROMatlabCRSpeedupMany{} times the transferred RO-CR-P-LPN time.  RO-CR-P-LPN
is faster for every sequence median.  Across 108 sequences the ratios are 1.38 and 1.20; they fall to
1.10 and 0.96 under low generator overlap and to
\ROPortfolioJuliaCRSpeedup{} and \ROPortfolioMatlabCRSpeedup{} on the
49-industry sequence.  Transfer is therefore useful on the prespecified
many-row sequences, but weak overlap erodes the gain and the portfolio sequence
produces slowdowns.

\subsection{Evidence Summary}

\Cref{fig:reuse-regime} places the within-language Phase-II results on one
scale; the primary and secondary rows use different baselines, and we do not
rank times across languages.  The evidence supports a deliberately narrow
conclusion.  Corral repair often saves substantial inner Wolfe work, especially
when generator scans are expensive and useful generators persist, but neither
repair nor transfer is uniformly faster.  The end-to-end results show that the
gain can survive initialization; the boundary studies show where it does not.

\begin{figure}[H]
\centering
\resizebox{0.98\textwidth}{!}{\begin{tikzpicture}
\begin{axis}[
  width=0.86\textwidth,
  height=1.75in,
  xmin=0.70,
  xmax=2.10,
  ymin=0.5,
  ymax=8.5,
  xtick={0.75,1.00,1.25,1.50,1.75,2.00},
  ytick={1,2,3,4,5,6,7,8},
  yticklabels={Klee--Minty,Portfolio,Low generator overlap,Broad objective grid,
    Prespecified sequences,Application-derived endpoints,Prespecified many-row,Evaluation components},
  xlabel={Baseline time / reuse time},
  xmajorgrids=true,
  grid style={black!12},
  axis line style={black!55},
  tick label style={font=\scriptsize},
  label style={font=\scriptsize},
  yticklabel style={font=\scriptsize,align=right},
  legend style={at={(0.99,0.98)},anchor=north east,legend columns=-1,
    draw=none,fill=white,fill opacity=0.85,text opacity=1,font=\scriptsize},
  clip=false
]
\addplot[forget plot,black!60,dashed,thick] coordinates {(1,0.5) (1,8.5)};
\addplot[only marks,mark=*,mark size=2.4pt,blue!65!black]
  coordinates {(0.94,1) (0.90,2) (1.10,3) (1.38,4)
    (1.85,5) (\JuliaRealCRSpeed,6) (2.03,7) (1.62,8)};
\addlegendentry{Julia}
\addplot[only marks,mark=square*,mark size=2.3pt,orange!80!black]
  coordinates {(0.75,1) (0.79,2) (0.96,3) (1.20,4)
    (1.48,5) (\MatlabRealCRSpeed,6) (1.46,7) (1.55,8)};
\addlegendentry{MATLAB}
\end{axis}
\end{tikzpicture}}
\caption{Reuse benefit by setting.  Values above one favor reuse.
\label{fig:reuse-regime}}
\begin{minipage}{0.98\textwidth}
\scriptsize Primary ratios are P-LPN/CR-P-LPN; objective-sequence and portfolio
ratios are CR-P-LPN/RO-CR-P-LPN.  Klee--Minty is inverted from CR-P-LPN/P-LPN.
Values below one are slowdowns; Netlib is not a reuse/no-reuse comparison.
\end{minipage}
\end{figure}
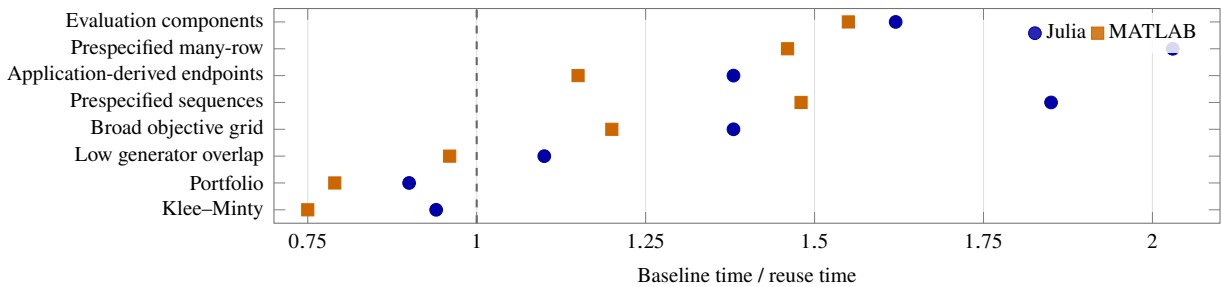

\FloatBarrier
\section{Discussion and Solver Management}\label{sec:discussion}

The experiments suggest a practical rule of thumb: reuse is attractive when a
full generator scan is expensive and the preceding corral is likely to remain
informative.  We do not yet have a reliable instance-level predictor, so the
implementation uses a fixed reuse limit \(B\).  Learning an adaptive rule from
the observed repair history is a natural next step.  Transfer across objectives
is more demanding because it also requires the hull to remain unchanged and
the relevant generators to persist.

\subsection*{Limitations}

The timing evidence covers three designed many-row problems, 19
application-derived LPs, and ten retained Netlib models.  This is not a
representative benchmark library and cannot establish broad solver
superiority.

P-LPN and CR-P-LPN assume a known interior feasible point.  The end-to-end
procedure supplies artificial Phase I, but all 19 application-derived LPs are
full dimensional; they therefore say little about the cost of affine-hull
reduction.  That reduction uses a numerical SVD and does not yet reconstruct a
complete original-space dual multiplier.  Finally, the combinatorial bound for
arbitrary real data is exponential.  The rational-data outer count is
polynomial in encoding length, but no polynomial bound is known for the inner
Wolfe algorithm.

\section{Conclusion}\label{sec:conclusion}

This study puts LP-Newton on a compact hull of normalized inequality rows.  The
radial endpoint recovers a finite optimum or proves unboundedness, and for
rational data P-LPN needs \(O(n^2\mathcal L)\) outer iterations.  Corral repair
changes how each projection starts, not which projection is computed; CR-P-LPN
therefore inherits the same outer sequence and bound in exact arithmetic.  No
corresponding polynomial bound is claimed for inner Wolfe work or total time.

Empirically, repair is most effective when checking all generators is costly
and successive corrals share useful generators.  In the end-to-end study of 19
application-derived LPs, CR-P-LPN has lower median total time than the faster
cold HiGHS mode (dual simplex or the interior-point method with crossover) on
18 workloads per language.  The retained Netlib results limit solver-dominance
claims, while the Klee--Minty and objective-transfer results show that reuse can
lose.  Two questions remain open: how to bound the inner Wolfe work and how to
predict from the current projection geometry whether reusing a corral will
save time.

\appendix
\section{Outer-Iteration Counting Details}\label{app:outer-details}

The notation is the same as in Sections~\ref{sec:formulation}--\ref{sec:corralreuse}.  Vectors are bold italic,
matrices are ordinary uppercase italic, and \(\vct{0}\) is the zero vector.
At outer iteration $i$, let
$\vct{y}_i=\rho_i\vct{c}$, let $\vct{z}_i$ be its Euclidean
projection onto the row-polar hull $\Q$, and define
\[
 \vct{d}_i:=\vct{y}_i-\vct{z}_i,\qquad
 h_i:=\max_{\vct{q}\in\Q}\vct{d}_i^T\vct{q}
 =\vct{d}_i^T\vct{z}_i.
\]
The next target is
$\vct{y}_{i+1}=\rho_{i+1}\vct{c}$, where
$\rho_{i+1}:=h_i/(\vct{d}_i^T\vct{c})$.  The terminal Wolfe corral $C_i$ is
affinely independent and represents
$\vct{z}_i=\sum_{j\in C_i}\lambda_j\vct{v}_j$ with positive weights summing
to one.

At a positive terminal target, \(\mathrm{VerifyFinite}\) forms
\(\vct{x}^*:=\vct{d}_{i-1}/h_{i-1}\) and
\(u_j^*:=\lambda_j/\rho_i\), (j=1,\ldots,m), and applies the normalized tests.  The recession
check tests
\(\vct{r}:=\vct{d}/(\vct{c}^T\vct{d})\), returning \(\mathrm{Unbounded}\)
when \(A\vct{r}\leq\vct{0}\) and \(\vct{c}^T\vct{r}=1\), and
\(\mathrm{RecTestFailed}\) otherwise.  Appendix~\ref{app:verification} maps these normalized outputs
to the original variables and gives the experiment checks.

Theorem~\ref{thm:finite} proves directly that the map from a nonterminal outer
iteration to its terminal Wolfe corral is injective.  The proof uses the
supporting hyperplane at the projection, strict radial decrease, and symmetry of the
Euclidean inner product.  The resulting complementary, data-independent count
is
\[
 K(N,n):=1+\sum_{j=1}^{\min\{n+1,N\}}\binom Nj.
\]
Indeed, a hull with \(N=m+1\) listed generators has at most
\(K(N,n)-1\) nonempty affinely independent subsets of size at most \(n+1\).
These bound the nonterminal corrals, and the positive-endpoint branch can make
one final projection with its target in \(\Q\).  Thus P-LPN makes at most
\(K(N,n)\) completed outer iterations.

Theorem~\ref{thm:rational-outer-bound} gives the sharper rational-data count
\(O(n^2\mathcal L)\), hence \(O(\mathcal L^3)\), where \(\mathcal L\) is the
row-polar binary input length.  Its proof separates three ingredients: the
Newton-halving inequality for the ray-distance function, the exact linear
distance segment next to the radial endpoint, and determinant bounds for the
endpoint LP and tangent-cone Gram system.  Theorem~\ref{thm:corralreuse} shows that
verified CR-P-LPN accepts the same unique projection at every target, so it
inherits both outer-iteration counts.  Neither count includes inner Wolfe work
or total time; an unverified warm attempt makes no radial update.
Thus, following the historical motivation in Section~\ref{sec:related}, a strongly
polynomial bound for the complete LP-Newton method remains an open direction.

\section{Projection and Verification Choices: Details}\label{app:projection-details}

Section~\ref{sec:validity-conditions} selects Euclidean projection, corral repair, and global
Wolfe verification, and summarizes the weighted metric and candidate-first
ordering discussed by \citet[Section 5]{Fujishige2019}.  This section gives the
details and comparison variants.  Nonzero-residual projections pass global
verification; zero residual identifies the target as its own projection.

Let \(\vct{v}_0:=\vct{0}\), \(\vct{v}_i:=\vct{a}_i\) for \(i\geq1\), and
\(V:=[\vct{v}_0,\ldots,\vct{v}_m]\).  For an active index set \(C\), the
Euclidean affine minimizer solves
\begin{equation}\label{eq:kkt}
 \min_{\vct{\lambda}}\ \tfrac12
 \norm{V_C\vct{\lambda}-\vct{y}}^2
 \quad\text{s.t.}\quad \mathbf{1}^T\vct{\lambda}=1.
\end{equation}
With \(\vct{z}:=V_C\vct{\lambda}\) and
\(\vct{d}:=\vct{y}-\vct{z}\), global optimality is equivalent to
\[
 \vct{d}^T\vct{v}_i\leq\vct{d}^T\vct{z}
 \quad(i=0,\ldots,m),
\]
with equality for every positive weight.  Wolfe alternates generator insertions
with \eqref{eq:kkt}; minor steps delete coefficients that reach zero.  A zero
residual proves that the target lies in \(\Q\); otherwise every returned
projection must pass the displayed test.  If a warm computation reaches a
numerical or resource limit before verification, it returns no projection;
cold Wolfe recomputes it.  The code uses a rank-revealing KKT alternative and
rejects numerical duplicate generators.

The early-support comparison forms a globally valid separating support from a
provisional inner point, and the Frank--Wolfe comparison uses exact segment
minimization.  Both are capped and followed by exact Wolfe and the final LP
tests.  The former is related to the Naive Separation Algorithm of
\citet{Matsuno2026}; the computational archive contains the complete formulas and code.

For the metric variant, let \(I\) be the \(n\)-by-\(n\) identity and define
\(P_\parallel:=\vct{c}\vct{c}^T\),
\(P_\perp:=I-P_\parallel\), and
\(M_\alpha:=P_\perp+\alpha P_\parallel\), with the fixed comparison value
$\alpha=0.1$.  At \(\vct{y}=\rho\vct{c}\), let
\(\vct{z}_\alpha:=\arg\min_{\vct{q}\in\Q}
\tfrac12(\vct{y}-\vct{q})^TM_\alpha(\vct{y}-\vct{q})\), and define
\(\vct{d}_\alpha:=M_\alpha(\vct{y}-\vct{z}_\alpha)\) and
\(h_\alpha:=\vct{d}_\alpha^T\vct{z}_\alpha\).  The projection condition is
\(\vct{d}_\alpha^T(\vct{q}-\vct{z}_\alpha)\leq0\) for every
\(\vct{q}\in\Q\).  When \(\vct{d}_\alpha^T\vct{c}>0\), its supporting
hyperplane gives the next radius
\(\rho_\alpha^+:=h_\alpha/(\vct{d}_\alpha^T\vct{c})\).

For the radial endpoint \(\vct{q}^*:=\rho^*\vct{c}\), the update is
exactly \(\rho^*\) when \(\vct{q}^*\) maximizes
\(\vct{d}_\alpha^T\vct{q}\) over \(\Q\).  Projection minimality, with
\(\Delta:=\rho-\rho^*>0\), gives
\(\norm{P_\perp\vct{z}_\alpha}_2^2+
\alpha(\rho-\vct{c}^T\vct{z}_\alpha)^2\leq\alpha\Delta^2\), so
\(\vct{z}_\alpha\to\vct{q}^*\) as \(\alpha\downarrow0\).  Finiteness of the
faces of \(\Q\) gives an exact endpoint update for every sufficiently small
positive \(\alpha\).  This is a comparison result, not the default design.

The candidate set \(\mathcal J_k\) contains the repaired corral and may add
generators that were support-tight in two consecutive outer iterations.
Testing it first changes only the verification order; the no-repair control
uses the same order without repair.  Every comparison variant retains its
stated cap, global Wolfe verification, and the final LP tests.

The dual minimum-norm, nearest-pair, and affine-restriction methods cited in
Section~\ref{sec:related} alter the nearest-point problem.  CR-P-LPN instead changes
only Wolfe's initial state for the same Euclidean projection onto \(\Q\), then
verifies globally.

\subsection{Implicit operations and corral repair}

The code stores the normalized matrix \(A\), not
\(V=[\vct{0},\vct{a}_1,\ldots,\vct{a}_m]\).  It evaluates
\[
 V\vct{\lambda}
 =A^T\vct{\lambda}_{1:m},\qquad
 V^T\vct{d}=\begin{bmatrix}0\\
 A\vct{d}\end{bmatrix},
\]
where \(\vct{\lambda}_{1:m}\) excludes the weight of the origin.  Only
current-corral columns are materialized.

For a new target \(\vct{y}^+\),
\(\mathrm{Repair}(\vct{y}^+;C,\vct{\lambda})\) starts from the preceding
corral and its positive weights.  Let \(\vct{\lambda}^{\rm aff}\) be the
solution of \eqref{eq:kkt} at \(\vct{y}^+\).  If it is not positive, repair
uses the ordinary Wolfe minor step
\[
 \theta:=\min\left\{1,
 \min_{i\in C:\lambda_i>\lambda_i^{\rm aff}}
 \frac{\lambda_i}{\lambda_i-\lambda_i^{\rm aff}}\right\},
 \qquad
 \vct{\lambda}^{+}:=(1-\theta)\vct{\lambda}
 +\theta\vct{\lambda}^{\rm aff}.
\]
When \(\theta<1\), the point
\(V_C\vct{\lambda}^{+}\) is the first boundary point on the segment toward
the affine minimizer, and at least one weight is zero.  Repair deletes every
zero-weight generator and resolves \eqref{eq:kkt}.  At most \(|C|-1\) such
deletion passes occur.  The returned pair
\((C^0,\vct{\lambda}^0)\) is only \(\mathrm{WarmWolfe}\)'s initial state,
not an accepted projection.

\subsubsection{Optional exact-reuse interval}

\begingroup
\small
\begin{theorem}[Exact-reuse interval]\label{thm:supp-exact-reuse}
Fix a nonempty affinely independent generator set \(C\).  For \(\rho\geq0\), put
\(P:=V_C\) and define \(\vct{\lambda}_C(\rho)\), \(\nu(\rho)\), and
\(\vct{z}_C(\rho):=P\vct{\lambda}_C(\rho)\) by
\[
 \begin{bmatrix}P^TP&\mathbf1\\\mathbf1^T&0\end{bmatrix}
 \begin{bmatrix}\vct{\lambda}_C(\rho)\\\nu(\rho)\end{bmatrix}
 =\begin{bmatrix}\rho P^T\vct{c}\\1\end{bmatrix}.
\]
With \(\vct{v}_j\), \(j=0,\ldots,m\), denoting all generators, define
\[
 \mathcal I_C:=\{\rho\geq0:\vct{\lambda}_C(\rho)\geq\vct{0},\
 (\vct{z}_C(\rho)-\rho\vct{c})^T
 (\vct{v}_j-\vct{z}_C(\rho))\geq0\ (j=0,\ldots,m)\}.
\]
Then \(\mathcal I_C\) is an interval, possibly empty or unbounded.  For every
\(\rho\in\mathcal I_C\), \(\vct{z}_C(\rho)\) is
\(\operatorname{Proj}_{\Q}(\rho\vct{c})\); if all weights are
positive, \(C\) is its Wolfe corral.
\end{theorem}

\begin{proof}
Affine independence makes the KKT matrix nonsingular, so the weights and
\(\vct{z}_C\) are affine in \(\rho\).  For any \(r\in C\), KKT orthogonality gives
 \[
 (\vct{z}_C-\rho\vct{c})^T
 (\vct{v}_j-\vct{z}_C)
 =
 (\vct{z}_C-\rho\vct{c})^T(\vct{v}_j-\vct{v}_r),
\]
which is affine in \(\rho\).  Hence \(\mathcal I_C\) is an intersection of
affine half-lines.  Its weight conditions put \(\vct{z}_C\) in \(\Q\); the
other inequalities are exactly the first-order conditions for Euclidean
projection.

\end{proof}

The interval identifies the target radii for which a fixed corral determines the projection, but Algorithm~\ref{alg:corralreuse} neither computes nor requires it.
\endgroup

Figure~\ref{fig:corral-repair-mechanism} separates this affine repair from acceptance of the new
projection.  At outer iteration \(k\), \(\mathcal J_k\) contains the repaired
corral and at most \(\max\{|C^0|,\lfloor\eta N\rfloor\}\) generators.
\(\mathrm{WarmWolfe}\) tests these generators first and then the remaining
generators.  A violation causes an ordinary Wolfe major step; a numerical or
resource-limit exit before verification triggers cold recomputation at the
same target.  Only iterations two through \(B\) may start warm, and no radial
update is made until the projection is verified.

\subsection{Reuse across changing objective vectors}\label{app:objective-transfer}

For the normalized sequence \((\mathrm P_t)\) in Section~\ref{sec:objective-reuse}, all
objectives share \(A\) and \(\Q\).  Repeated-objective CR-P-LPN (RO-CR-P-LPN)
solves the first objective cold and offers its terminal corral and weights to
the first projection for the next objective.  It transfers no state after an
unbounded or failed status.  Repair, global Wolfe verification, and the output checks in Appendix~\ref{app:verification} remain mandatory; only Wolfe's initial state changes.
The computational archive implements this rule in both languages.

\section{Optional Two-Stage Initialization}\label{app:initialization}
\label{sec:supp-phase-one}

The optional two-stage construction supplies the strictly feasible point used
by the normalization in Section~\ref{sec:formulation}.  It solves an artificial Phase-I LP before
the original objective in Phase II; the end-to-end experiment includes both.
For original data \(A^0,\vct{b}^{\,0}\), choose any
\(\vct{x}^{\rm start}\in\R^n\), put
\[
 \theta^0:=\min\{0,\min_i(b_i^0-(\vct{a}_i^0)^T
 \vct{x}^{\rm start})\}-1,
\]
and solve
\begin{equation}\label{eq:supp-phase-one}
 \max_{\vct{x}^{\,0},\theta}\{\theta:
 A^0\vct{x}^{\,0}+\theta\mathbf1\leq\vct{b}^{\,0},\ \theta\leq1\}
\end{equation}
from the strictly feasible point \((\vct{x}^{\rm start},\theta^0)\).  Its finite
optimum is attained; let \((\vct{x}^{\,0,{\rm I}},\theta^*)\) be an optimizer.  The optimal
value \(\theta^*\) is positive exactly when the original system has an ambient
strict point, negative exactly when it is infeasible, and zero when it is
feasible with empty ambient interior.

The definition of \(\theta^0\) makes every auxiliary inequality strict, so
P-LPN applies.  Feasibility of \((\vct{x}^{\,0},0)\) links
\(\theta^*\geq0\) to original feasibility, while \(\theta^*>0\) is equivalent
to positive slack in every original row.  These observations prove the three
sign conclusions; in the positive case, \(\vct{x}^{\,0,{\rm I}}\) is the Phase-II
interior point.

The zero case is reduced to its affine hull.  For a current system
\(B\vct{z}\leq\vct{d}\), optimal row multipliers \(\vct{\lambda}\geq\vct{0}\)
at \(\theta^*=0\) satisfy
\[
 B^T\vct{\lambda}=\vct{0},\qquad
 \mathbf1^T\vct{\lambda}=1,\qquad
 \vct{d}^T\vct{\lambda}=0.
\]
Consequently every row with \(\lambda_i>0\) is tight at every feasible point.
Parameterizing the null space of these rows preserves the feasible set and
lowers its dimension; zero rows that reduce to \(0\leq0\) are deleted.  Each
zero-margin pass lowers the dimension or deletes a row, so the exact procedure
terminates after at most \(n\) dimension reductions and \(m\) such deletions.

\begin{algorithm}[!t]
\caption{Artificial Phase I with affine-hull reduction}
\label{alg:supp-phase-one}
\small
\begin{algorithmic}[1]
\Require \(A^0,\vct{b}^{\,0}\), any \(\vct{x}^{\rm start}\); maintain
\(\vct{x}^{\,0}=\vct{x}_c^{\,0}+N\vct{z}\)
\State \(\vct{x}_c^{\,0}\gets\vct{0}\), \(N\gets I_n\),
\(R\gets\{1,\ldots,m\}\), \(\vct{z}^0\gets\vct{x}^{\rm start}\)
\Loop
  \If{\(R=\varnothing\)} \State \Return \((\vct{x}_c^{\,0},N,R)\) \EndIf
  \State \(B\gets (A^0)_RN\),
  \(\vct{d}\gets\vct{b}^{\,0}_R-(A^0)_R\vct{x}_c^{\,0}\)
  \State \(\theta^0\gets\min\{0,\min_i(d_i-B_{i:}\vct{z}^0)\}-1\)
  \State Solve Equation~\eqref{eq:supp-phase-one} in the current coordinates, obtaining
  \((\vct{z}^{\rm I},\theta^*)\) and row multipliers \(\vct{\lambda}\)
  \If{\(\theta^*<0\)} \State \Return \(\mathrm{Infeasible}\) \EndIf
  \If{\(\theta^*>0\)}
    \State \Return \((\bar{\vct{x}}^{\,0},N,R)\), where
    \(\bar{\vct{x}}^{\,0}\gets\vct{x}_c^{\,0}+N\vct{z}^{\rm I}\)
  \EndIf
  \State \(\mathcal T\gets\{i\in R:\lambda_i>0\}\), \quad
  \(\vct{x}_c^{\,0}\gets\vct{x}_c^{\,0}+N\vct{z}^{\rm I}\)
  \If{\(\operatorname{rank}(B_{\mathcal T})>0\)}
    \State Choose \(M\) spanning \(\ker(B_{\mathcal T})\), \quad \(N\gets NM\)
  \Else
    \State \(R\gets R\setminus\mathcal T\)
  \EndIf
  \State Delete constant true rows from \(R\), \quad \(\vct{z}^0\gets\vct{0}\)
\EndLoop
\end{algorithmic}
\end{algorithm}

At return, the retained system in coordinates
\(\vct{x}^{\,0}=\bar{\vct{x}}^{\,0}+N\vct{w}\) has \(\vct{w}=\vct{0}\)
strictly feasible.  The Julia and MATLAB routines use
the same procedure with original-scale outcome checks and an SVD rank threshold.
The threshold is suitable for the reported data but does not determine the exact
rank for arbitrarily ill-conditioned inputs.

\section{Comparator and Prespecified Parameters}\label{app:comparators}

The author-implemented simplex code receives the same interior point and
positive diagonal scaling \(D\in\R^{n\times n}\) as P-LPN.  With
\(\vct{s}^{\,0}=\vct{b}^{\,0}-A^0\bar{\vct{x}}^{\,0}>\vct{0}\), the code splits the scaled
coordinate as \(\vct{\xi}=\vct{\xi}^+-\vct{\xi}^-\), with
\(\vct{\xi}^+,\vct{\xi}^-\geq\vct{0}\), and uses the standard-form
matrix \([A^0D\;-A^0D\;I_m]\), where \(I_m\) is the \(m\)-by-\(m\) identity.
The slack columns give the initial basis.  The code
uses Dantzig's rule, deterministic ratio-test ties, Bland's rule after 50
consecutive degenerate pivots, and refactorization every 250 pivots.  Every
answer passes the applicable original-scale primal, dual, nonnegativity, and
gap tests.  Both source files are in the computational archive.  This author-implemented
code is a similar-implementation peer check.  The
primary method comparison is P-LPN versus CR-P-LPN, whose model, Wolfe code,
language, stopping rules, and verification tests are shared.  HiGHS is the external
practical reference and is not used to attribute a time difference to corral
repair.

Testing \(\alpha\in\{0.05,0.10,0.20,0.50,1,2\}\) on 16 development problems
fixes \(\alpha=0.1\) for the metric comparison.  On dimension-50 Klee--Minty,
it reduces the outer and affine counts from 5/23 to 4/22, but this does not
establish a general runtime improvement.

The reported CR-P-LPN experiments use \(B=4\), \(\eta=0.5\), and
\(\alpha=1\).  Correctness holds for every finite \(B\); increasing \(B\)
offers more reuse opportunities but can add repair and verification work.  We
do not claim that \(B=4\) is universally optimal.

\section{Floating-Point Verification}\label{app:verification}

Algorithms~\ref{alg:plpn} and~\ref{alg:corralreuse} return normalized \(\vct{x}^*,\vct{u}^*,\vct{r}\).  With
\(S:=\operatorname{Diag}(\vct{s}^{\,0})\), map them to the original scale by
\[
 \vct{x}^{\,0*}:=\bar{\vct{x}}^{\,0}+\vct{x}^*,\qquad
 \vct{u}^{\,0*}:=\norm{\vct{c}^{\,0}}S^{-1}\vct{u}^*,\qquad
 \vct{r}^{\,0}:=\vct{r}/\norm{\vct{c}^{\,0}}.
\]
Under diagonal equilibration, primal and recession displacements are first
mapped through \(D\), and the dual scaling factor is
\(\norm{D\vct{c}^{\,0}}\).

Bounded outputs with reconstructed original-system multipliers are checked in
the original variables for primal and dual residuals, nonnegativity, objective
gap, and radial agreement; reduced cases use the applicable primal and objective
checks described below.  For a vector, \((\cdot)_+\) denotes the componentwise positive
part; \(\norm{\cdot}_\infty\) and \(\norm{A^0}_\infty\) are the vector and
induced matrix infinity norms.  Define
\begin{align*}
 \eta_p&:=\frac{\norm{(A^0\vct{x}^{\,0}-\vct{b}^{\,0})_+}_\infty}
 {\max\{1,\norm{\vct{b}^{\,0}}_\infty,
 \norm{A^0}_\infty\max(1,\norm{\vct{x}^{\,0}}_\infty)\}},\\
 \eta_d&:=\frac{\norm{(A^0)^T\vct{u}^{\,0}-\vct{c}^{\,0}}_\infty}
 {\max\{1,\norm{\vct{c}^{\,0}}_\infty,
 \norm{A^0}_\infty\max(1,\norm{\vct{u}^{\,0}}_\infty)\}},\\
 \eta_g&:=\frac{|(\vct{b}^{\,0})^T\vct{u}^{\,0}-(\vct{c}^{\,0})^T\vct{x}^{\,0}|}
 {\max\{1,|(\vct{b}^{\,0})^T\vct{u}^{\,0}|,
 |(\vct{c}^{\,0})^T\vct{x}^{\,0}|\}}.
\end{align*}
The first two must not exceed $10^{-6}$, and the gap and radial agreement must
not exceed $10^{-5}$; $\vct{u}^{\,0}\geq-10^{-6}$.  An unbounded answer is
normalized to $(\vct{c}^{\,0})^T\vct{r}^{\,0}=1$ and is returned as \texttt{unbounded} only
if
\[
 \eta_r:=\frac{\norm{(A^0\vct{r}^{\,0})_+}_\infty}
 {\max\{1,\norm{A^0}_\infty\norm{\vct{r}^{\,0}}_\infty\}}\leq10^{-6}.
\]
Otherwise, the status contains \texttt{inaccurate}.  Both implementations
independently apply the applicable original-variable tests.

For the expanded application-derived experiment, the production tolerance
is \(10^{-7}\) for \(\eta_p\), \(\eta_d\), and dual nonnegativity, and
\(10^{-6}\) for \(\eta_g\) and radial agreement.  P-LPN and CR-P-LPN rows
store status, a verification flag, objective error, \(\eta_p\), and the unscaled
dual residual, not every normalized field.  A row is verified only when all
production checks pass.  The two stored error fields and all separately stored
HiGHS verification fields are also tested at \(10^{-7}\).

At a finite endpoint, an optional polishing step uses the positive dual
support \(\mathcal S\).  Here \((A^0)_{\mathcal S}\) and
\(\vct{b}^{\,0}_{\mathcal S}\) retain the rows indexed by \(\mathcal S\),
\(\mathcal S^c\) is its complement, and \({}^\dagger\) is the Moore--Penrose
pseudoinverse.  If \(|\mathcal S|\geq n\) and \((A^0)_{\mathcal S}\) has
full column rank, the step forms
\[
 \widetilde{\vct{x}}^{\,0}:=((A^0)_{\mathcal S})^\dagger
 \vct{b}^{\,0}_{\mathcal S},
 \qquad
 \widetilde{\vct{u}}^{\,0}_{\mathcal S}
 :=(((A^0)_{\mathcal S})^T)^\dagger\vct{c}^{\,0},
 \qquad
 \widetilde{\vct{u}}^{\,0}_{\mathcal S^c}:=\vct{0}.
\]
Multipliers that are negative beyond the feasibility tolerance reject this
candidate.  Tiny negative values are set to zero.  The candidate is retained
only when the maximum normalized primal, dual, nonnegativity, and gap residual
decreases.  It must then pass the full original-scale tests.  The raw
files report whether polishing was used.

The author-implemented simplex code uses the same \(\eta_p,\eta_d,\eta_g\) and
nonnegativity tests after undoing its shift and column scaling.  Its dual vector
is the equality-basis multiplier \(\vct{\pi}\), whose nonnegativity follows
from the reduced costs of the slack columns.  Thus both methods' optimal
outputs satisfy the same primal-dual requirements; radial agreement is checked
additionally for the row-polar methods.

Every final result is checked in original coordinates for primal feasibility
and objective agreement, or for recession feasibility and improvement.
Original-scale dual, nonnegativity, and gap tests exclude affine-hull-reduced
cases because the pipeline does not reconstruct original-system multipliers.

\section{Additional Numerical Evidence}\label{app:numerical-details}

\subsection{Phase-II many-row and end-to-end results}

All 120 Phase-II runs satisfy the stated verification tests; Julia and MATLAB operation counts agree in every
method--instance cell.  The maximum relative objective error and
normalized dual residual are \(1.49\times10^{-13}\) and
\(5.04\times10^{-14}\).  Component controls are consistent with terminal-corral
reuse, rather than candidate-first ordering alone, as the useful component.

The end-to-end experiment has \EToETimedRows{} timed rows
(\EToERowsPerSolver{} each for Julia, MATLAB, and HiGHS); 90 cover the three
many-row instances below.  All satisfy their original-scale tests; maximum
cross-solver relative objective disagreement is
\(\EToEMaxObjectiveDisagreement\).  Both languages reduce all
\EToEBoundaryCases{} zero-margin systems consistently with cold HiGHS, and
\EToERegressionRows{} untimed rows cover further boundary cases.

\begin{table}[ht]
\caption{End-to-end timing components on the prespecified many-row instances.
\label{tab:supp-many-row}}
\centering
\small\centering\begin{tabular}{lrrrrrrrrr}
\toprule
& & & \multicolumn{3}{c}{P-LPN} & \multicolumn{3}{c}{CR-P-LPN} & \\
\cmidrule(lr){4-6}\cmidrule(lr){7-9}
Language & $n$ & $m$ & Phase I & Phase II & Total & Phase I & Phase II & Total & H-fast \\
\midrule
Julia & 16 & 5,000 & 0.0023 & 0.0083 & 0.0101 & 0.0024 & 0.0048 & 0.0074 & 0.0395 \\
Julia & 32 & 25,000 & 0.0116 & 0.0806 & 0.0926 & 0.0138 & 0.0508 & 0.0657 & 0.3016 \\
Julia & 64 & 100,000 & 0.1221 & 1.6996 & 1.8109 & 0.1033 & 0.9962 & 1.1394 & 1.9564 \\
MATLAB & 16 & 5,000 & 0.0059 & 0.0165 & 0.0231 & 0.0065 & 0.0100 & 0.0173 & 0.0395 \\
MATLAB & 32 & 25,000 & 0.0288 & 0.1225 & 0.1572 & 0.0295 & 0.0687 & 0.0998 & 0.3016 \\
MATLAB & 64 & 100,000 & 0.1926 & 1.7193 & 1.8949 & 0.2232 & 0.7909 & 1.0332 & 1.9564 \\
\bottomrule
\end{tabular}\ \null\par\smallskip
\begin{minipage}{0.98\textwidth}
\footnotesize Times are median seconds over five randomized-order,
single-thread runs after one unrecorded run.  Total is measured directly; the
separate component medians need not sum exactly to it.  H-fast is the smaller
median of cold HiGHS dual simplex and cold HiGHS IPM.  The saved files already
have \(\vct{b}=\vct{1}\) and strict point \(\vct{x}=\vct{0}\), so forced
artificial Phase I is conservative for these instances.
\end{minipage}
\end{table}

Geometric-mean total P-LPN/CR-P-LPN ratios are
\EToEJuliaPCRGM{}/\EToEMatlabPCRGM{} in Julia/MATLAB; faster-cold-HiGHS/CR-P-LPN
ratios are \EToEJuliaHighsCRGM{}/\EToEMatlabHighsCRGM{}.  Both row-polar methods
are below H-fast in all six cells, though H-fast/P-LPN narrows to 1.08/1.03 on
the 100,000-row case.

On the 100,000-row problem, allocations are 97.66 MiB for the old normalized
container, 49.59 MiB for the earlier full scan, and 0.76 MiB for the implicit
scan.  Release \texttt{20260809r2} retains the raw records and generated tables available at that freeze.

\subsection{Secondary changing-objective results}

Twenty-four fixed-constraint sequences cross 12 objectives, three
generator-overlap levels, and \(n\in\{8,16,32,64\}\).  All
\ROPrimaryObjectiveSolves{} row-polar and \ROSimplexObjectiveSolves{} simplex
solves pass against known box optima.

For 18 sequences with at least 5,000 rows, separate-CR/RO-CR ratios are
\ROJuliaCRSpeedupMany{} [\ROJuliaCRSpeedupManyLower,\ROJuliaCRSpeedupManyUpper]
in Julia and
\ROMatlabCRSpeedupMany{} [\ROMatlabCRSpeedupManyLower,\ROMatlabCRSpeedupManyUpper]
in MATLAB, with 18/18 wins each.

The broader grid uses \(n\in\{16,32,64\}\), row factors 250, 750, and 1500,
nonzero densities .05 and .20, two seeds, three overlap levels, and five
objectives.  All 972 rows per language satisfy the stated verification tests; low-overlap MATLAB cells reverse
the average benefit (Table~\ref{tab:supp-broad-grid}).

\begin{table}[ht]
\caption{Broad changing-objective grid: separate CR-P-LPN/RO-CR-P-LPN time ratio.
\label{tab:supp-broad-grid}}
\centering
\renewcommand{\arraystretch}{0.88}\small\centering
\begin{tabular}{lrrrr}
\toprule
Generator overlap & Julia CR time/RO time & Julia wins & MATLAB CR time/RO time & MATLAB wins \\
\midrule
All & 1.38 & 99/108 & 1.20 & 82/108 \\
High & 1.63 & 36/36 & 1.43 & 36/36 \\
Medium & 1.45 & 36/36 & 1.27 & 36/36 \\
Low & 1.10 & 27/36 & 0.96 & 10/36 \\
\bottomrule
\end{tabular}\ \null\par\smallskip
\begin{minipage}{0.98\textwidth}
\footnotesize Entries are geometric means and sequence-median wins over three
randomized-order repetitions.
\end{minipage}
\end{table}

All \ROPortfolioObjectiveSolves{} solves satisfy the stated verification tests on
the fixed 14,400-row,
49-industry, 12-month portfolio.  Median P-LPN/CR-P-LPN/RO-CR-P-LPN times are
2.859/2.278/2.523 seconds in Julia and 3.214/3.577/4.512 in MATLAB.  Ratios
\ROPortfolioJuliaCRSpeedup{}/\ROPortfolioMatlabCRSpeedup{} show no transfer gain;
retained-state HiGHS is \ROPortfolioHighsSpeedup{} times faster than cold HiGHS.

\subsection{Primary performance boundaries and component evidence}

All 570 end-to-end rows on the 19 application-derived LPs (190 each for Julia,
MATLAB, and HiGHS) satisfy the stated verification tests.  Julia/MATLAB total P-LPN/CR-P-LPN ratios are
1.25 [1.17,1.34] (18/19 wins) and 1.09 [1.02,1.16] (12/19 wins);
H-fast/CR-P-LPN ratios are 4.93 and 3.06 (18/19 wins each), with Spambase the
common exception.

The separate 570-run Phase-II study also satisfies the verification tests: Julia/MATLAB P-LPN/CR-P-LPN
ratios are \JuliaRealCRSpeed/\MatlabRealCRSpeed{} (13/19 and 12/19 wins).  All
4,320 component runs satisfy the verification tests; pure-reuse, candidate-first, and full CR-P-LPN
ratios are 1.66/1.61, 0.98/0.97, and 1.62/1.55.

All 120 bounded-zonotope and 160 scaled-polar dimension-50 Klee--Minty runs
pass; the latter's maximum relative error is \(1.04\times10^{-8}\).
Baseline LP-Newton uses two projections and eleven affine/linear-optimization
steps; Julia/MATLAB CR-P-LPN/P-LPN ratios 1.06/1.34 show slower repair.  On ten
retained Netlib models, simplex wins 10/10 and 9/10 medians.

\section*{Data and Code Availability}
\begingroup\small
Computational materials for this version are available from the
corresponding author at
\href{mailto:shi@rs.tus.ac.jp}{\texttt{shi@rs.tus.ac.jp}}.
The archive contains the Julia and MATLAB implementations, instance
generators, raw results, and scripts used to produce and verify the
reported tables. Third-party software and datasets are identified in
the documentation but are not redistributed.
\endgroup

\bibliographystyle{plainnat}

\end{document}